\documentclass[12pt]{amsart}
\usepackage{amssymb,amsmath}
\usepackage[dvips]{graphicx}
\usepackage{pslatex}
\usepackage{amsmath,amscd}
\usepackage{amsthm}
\usepackage{overpic}
\usepackage{color}

\theoremstyle{plain} \newtheorem{thm}{Theorem}[section]
\newtheorem{cor}[thm]{Corollary} \newtheorem{prop}[thm]{Proposition}
\newtheorem{lemma}[thm]{Lemma} 

\newtheorem{question}[thm]{Question} 

\newtheorem*{namedtheorem}{\theoremname}
\newcommand{\theoremname}{testing}
\newenvironment{named}[1]{\renewcommand{\theoremname}{#1}\begin{namedtheorem}}{\end{namedtheorem}}

\theoremstyle{definition}

\newtheorem{ex}[thm]{Example} 
\theoremstyle{remark}

\newcommand{\Sth}{\mathbb{S}^3}

\begin{document}
	
\author{Christian Millichap}
\address{Department of Mathematics\\ 
	Furman University\\ 
	Greenville, SC 29613}
\email{christian.millichap@furman.edu}

\author{Rolland Trapp}
\address{Department of Mathematics\\ 
	California State University\\ 
	San Bernardino, CA 92407}
\email{rtrapp@csusb.edu}

\title{Symmetry groups of flat fully augmented links and their complements.}

\begin{abstract}
In this paper, we prove that the (orientation-preserving) symmetry groups of $b$-prime flat fully augmented links correspond exactly with the finite subgroups of $O(3)$. We accomplish this by first developing a dictionary between automorphisms of a  $3$-connected planar cubic graph associated to a flat fully augmented link $L$ and orientation-preserving symmetries of $L$. Our work also provides a simple method to explicitly construct infinite classes of distinct $b$-prime flat fully augmented links $\{L_i\}$ with $Sym^{+}(\Sth \setminus L_i) \cong Sym^{+}(\Sth, L_i) \cong G$, for any $G$ that is a finite subgroup of $O(3)$.
\end{abstract}

\maketitle

\section{Introduction}
\label{sec:intro}

Flat fully augmented links (flat FALs), and more generally, FALs are a well studied class of hyperbolic links that admit tractable geometric and topological structures. To construct a flat FAL $L$, start with a link $L' \subset \Sth$ with a fixed link diagram $D(L')$. Augment each twist region of $D(L')$ with a unknotted circle and undo all twists within each twist region to obtain a \textbf{flat FAL diagram} $D(L)$; see Figure \ref{fig:FlatFAL} for a diagram of a flat FAL.  The diagram $D(L)$ provides a partition of the components of $L$ into two classes: \textbf{crossing circles} which are the unknotted circles added via augmentation, and \textbf{knot circles} which are the components of $L$ corresponding to the original link $L'$. The work of Purcell \cite[Theorem 2.5]{Pu4} shows that an FAL $L$ is hyperbolic if and only if there exists a corresponding link diagram $D(L')$ that is non-splittable, prime, twist-reduced, and contains at least two twist regions where $D(L)$ is obtained from $D(L')$ by the augmentation process described above; see \cite{Pu4} for details on these diagrammatic definitions and background on FALs. Moving forward, we will assume that any FAL is hyperbolic. 

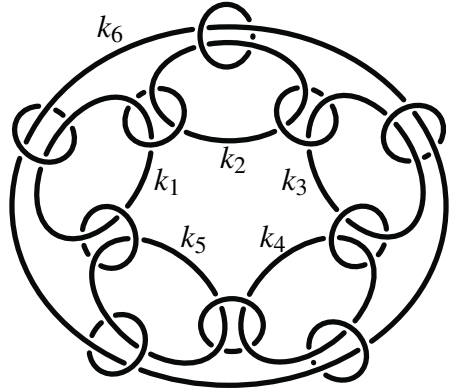
\begin{figure}[ht]
	\centering
	\begin{overpic}[scale=1.0]{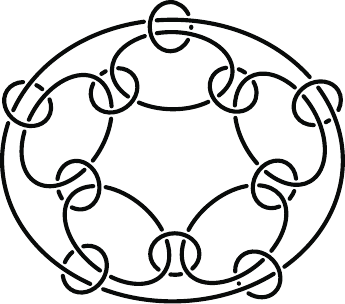}
		\put(33, 45){$k_1$}
		\put(62, 45){$k_3$}
		\put(39, 32){$k_5$}
		\put(57, 32){$k_4$}
		\put(48, 50){$k_2$}
		\put(20, 80){$k_6$}
			\end{overpic}
	\caption{A flat FAL with ten crossing circles and six knot circles. The knot circles are labeled by $k_i$, for $i =1, \ldots, 6$.}
	\label{fig:FlatFAL}
\end{figure}

Frequently, basic combinatorial structures associated with an FAL $L$  encode coarse, and sometimes even precise, geometric and topological data about $(\Sth, L)$ and $(\Sth \setminus L)$. For instance, the volume of any FAL is bounded in terms of linear functions in the number of crossing circles of $L$, and for certain infinite classes of FALs, the volume is precisely determined in this manner; see \cite[Proposition 3.6]{Pu4} and the appendix to \cite{La}. Thus, it is natural to consider what other geometric and topological features of a flat FAL $L$ are encoded in  combinatorial data associated to $L$.

Here, we are interested in investigating symmetries of flat FALs and their complements. Given a link $L$, we let $Sym(\Sth, L)$ denote the group of self-homeomorphisms of the pair $(\Sth, L)$, up to isotopy, and we let $Sym^{+}(\Sth, L)$ denote the index two subgroup of $Sym(\Sth, L)$ consisting of orientation-preserving symmetries. We define $Sym(\Sth \setminus L)$ and $Sym^{+}(\Sth \setminus L)$ in a similar fashion. Mostow--Prasad Rigidity implies that any symmetry of a hyperbolic $\Sth \setminus L$ corresponds with a unique isometry of its hyperbolic structure. This motivated the following question posed by Paoluzzi--Porti \cite{PaPo2009}: What is the relationship between symmetries of a hyperbolic link and isometries of its complement?   For any hyperbolic link $L \subset \Sth$, we always have $Sym^{+}(\Sth, L) \subseteq Sym^{+}(\Sth \setminus L)$, where both groups are finite. However, this set inclusion can be a strict containment \cite{HeWe1992} and can even result in an arbitrarily large index as subgroup containment \cite{MiTr2025}. There also exist infinite classes of hyperbolic links where $Sym^{+}(\Sth \setminus L) \cong Sym^{+}(\Sth, L)$; for instance, knots \cite{GL} and $2$-bridge links \cite{GuFu} have this property. Furthermore, while $Sym^{+}(\Sth, L)$ must be a finite subgroup of the special orthogonal group of dimension four, $SO(4)$, the work of Paoluzzi--Porti \cite[Theorem 1]{PaPo2009} shows that for any finite group $G$, there exists a hyperbolic link $L$ such that $G \cong Sym^{+}(\Sth \setminus L)$. In light of this wide variety of relations between symmetries of hyperbolic links and their complements, it is natural to ask what happens for other important classes of hyperbolic links, which motivated us to consider flat FALs.

Our work shows there is a basic combinatorial structure associated to a flat FAL $L$ that encodes many, and sometimes all, the orientation-preserving symmetries of $L$ and its complement. This structure is called a (painted) crushtacean graph $\Gamma$, which is a $3$-connected, planar, cubic graph with an assigned subset of ``painted'' edges that corresponds with a perfect matching of $\Gamma$. The graph $\Gamma$  can easily be constructed from a flat FAL diagram $D(L)$; see Section \ref{Sec:BackgroundCrush} and Figure \ref{fig:BuildCrush}. In what follows, we say that a symmetry  $\rho$ of a flat FAL $L$ (with a fixed flat FAL diagram) is \textbf{type-preserving} if $\rho$ maps knot circles to knot circles and maps crossing circles to crossing circles. In addition, $Aut_{p}(\Gamma)$ is the group of automorphisms of $\Gamma$ that preserves painted edges. 
 
 \begin{thm}
	\label{thm:main1}
	Let $L$ be a flat FAL with a painted crushtacean $\Gamma$. Then there exists a  monomorphism $\phi: Aut_{p}(\Gamma) \rightarrow Sym^{+}(\Sth, L)$. Furthermore, the image of $\phi$ only contains type-preserving symmetries of $L$. 
\end{thm}

Since $Sym^+(\Sth, L) \subseteq Sym^+(\Sth \setminus L)$, the map $\phi$ can also be interpreted as a monomorphism from $Aut_{p}(\Gamma)$ to $Sym^+(\Sth \setminus L)$. 

This monomorphism is quite useful since automorphisms of $3$-connected planar graphs are well studied objects. Whitney's Theorem \cite{WH1932} shows that a $3$-connected planar graph $\Gamma$ has a unique embedding on the $2$-sphere and Mani's Theorem \cite{Mani1971} shows that this unique embedding of $\Gamma$ can be realized on the $2$-sphere so that each automorphism extends to an isometry of the $2$-sphere. As a result, $\phi(Aut_{p}(\Gamma)) \leq Sym^{+}(\Sth, L)$ must be isomorphic to a finite subgroup of $O(3)$, the orthogonal group of dimension $3$.  However, the index $[Sym^{+}(\Sth, L): \phi(Aut_{p}(\Gamma))]$ can be arbitrarily large; see Section \ref{subsec:On}. Therefore, it would help to know when $\phi$ is onto, which happens under two particular geometric conditions on $\Sth \setminus L$, which we now describe.

 We say that a flat FAL $L$ is \textbf{$b$-composite} if $\Sth \setminus L$ it admits a separating pair of totally geodesic thrice-punctured spheres, which by cutting along and re-gluing their boundaries in pairs, decomposes $L$ into a pair of flat FALs; see Figure \ref{fig:BSD}. If $L$ is not equivalent to the Borromean rings and does not admit such a decomposition, then we say that $L$ is \textbf{$b$-prime}. The decomposition of FAL complements into a set of $b$-prime FAL complements and Borromean ring complements was studied in \cite{MRSTZ2025}. In particular, Theorem 5.2 of \cite{MRSTZ2025} shows that an FAL is $b$-prime if and only if its corresponding painted crushtacean meets a simple combinatorial condition. This makes it straightforward to detect whether a flat FAL is $b$-composite, and then,  decompose it into $b$-prime and Borromean ring components. 
 
\begin{figure}[ht]
\begin{center}
\includegraphics{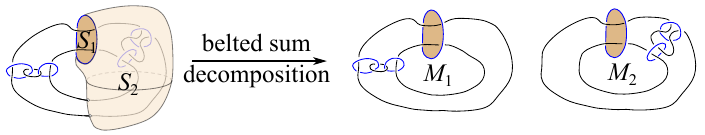}
\end{center}
	\caption{Decomposing a $b$-composite FAL into its $b$-prime summands.}
	\label{fig:BSD}
\end{figure}
 
 On the other hand, reflection in the projection plane for a flat FAL diagram $D(L)$ is a symmetry of $(\Sth, L)$. By Mostow-Prasad Rigidity, this reflection corresponds with an isometry of $\Sth \setminus L$ whose fixed-point set is a totally geodesic surface, which we call a \textbf{reflection surface} for $\Sth \setminus L$ or $D(L)$.   It is possible for $\Sth \setminus L$ to admit multiple reflection surfaces, though such flat FALs were classified by previous work of the authors in \cite[Theorem 1.2]{MiTr2023}. Together, the properties of being $b$-prime and having a unique reflection surface guarantee $\phi$ is onto. If $L$ also fails to be a signature link, which is a flat FAL with a very specific structure, then $\phi$ also captures the symmetries of the corresponding link complement. We refer the reader to \cite[Section 5]{MiTr2023} for specifications on signature links.

 \begin{cor}
 	\label{cor:main1}
 	Let $L$ be a $b$-prime flat FAL whose complement admits a unique reflection surface.  Then $Sym^+(\Sth, L) \cong Aut_{p}(\Gamma)$. If in addition $L$ is not a signature link, then $Sym^+(\Sth \setminus L) \cong Sym^+(\Sth, L) \cong Aut_{p}(\Gamma)$.
 \end{cor}

Corollary \ref{cor:main1} and our work in Section \ref{sec:MultRef} make it straightforward to determine $Sym^+(\Sth, L)$, for any $b$-prime flat FAL $L$. If $\Sth \setminus L$ admits a unique reflection surface, then consider its crushtacean $\Gamma$, which is unique in this case (see Corollary \ref{cor:bprimeunique}), and analyze $Aut_{p}(\Gamma)$. If $\Sth \setminus L$ admits multiple reflection surfaces, then $L$ is equivalent to either the Borromean rings or a link from one of two specific infinite classes of flat FALs. Such FALs and the symmetry groups of these links and their complements are classified in Section \ref{sec:MultRef}. Based on these two possibilities, we see that $Sym^{+}(\Sth, L)$ must be isomorphic to a finite subgroup of $O(3)$. In Section \ref{sec:classifying}, we give a combinatorial construction using planar $3$-connected graphs to show how to explicitly build a $b$-prime flat FAL with any possible desired symmetry group. In fact, we can modify our combinatorial construction to prove the following theorem.

\begin{thm}
	\label{thm:main2}
	Let $G$ be an abstract group. Then $G \cong Sym^{+}(\Sth, L)$ for some b-prime flat FAL $L$ if and only if $G$ is isomorphic to a  finite subgroup of $O(3)$. Furthermore, for every finite $G \leq O(3)$, there exist an infinite class of distinct b-prime flat FALs $\{L_i\}$ such that $Sym^{+}(\Sth \setminus L_i) \cong Sym^{+}(\Sth, L_i) \cong G$, for all $i$. 
\end{thm}

Our paper is organized as follows. In Section \ref{Sec:BackgroundCrush}, we review how to construct a crushtacean graph for a flat FAL and provide a complete graph-theoretic characterization of crushtaceans for (hyperbolic) flat FALs. In Section \ref{sec:symFALviaautocrush}, we prove Theorem \ref{thm:main1} and Corollary \ref{cor:main1}. In Section \ref{sec:classifying}, we develop a basic combinatorial construction that we call a cycle expansion on a $3$-connected planar graph in order to build $b$-prime flat FALs that meet the qualifications of Theorem \ref{thm:main2}. In Section \ref{sec:MultRef}, we analyze flat FALs with multiple reflection surfaces and determine the orientation-preserving symmetry groups for such links and their corresponding complements. In Section \ref{sec:Conclusion}, we provide a quick  proof of Theorem \ref{thm:main2} and share some remaining questions with regards to symmetries of flat FALs. 

We would like to that Furman University for financially supporting in-person research meetings for this project via a Furman Standard Grant.


\section{Crushtacean graphs for flat FALs}
\label{Sec:BackgroundCrush}

The \textbf{painted crushtacean} for a flat FAL diagram $D(L)$, denoted $\Gamma(L)$, is the painted cubic graph constructed by replacing each crossing circle with a painted edge, as shown in Figure \ref{fig:BuildCrush}. When it is clear from context, we abbreviate $\Gamma(L)$ as $\Gamma$. Crushtaceans of FALs were first introduced by Chesebro--DeBlois--Wilton in \cite{ChDeWi}.  Purcell \cite{Pu4} showed that $\Gamma(L)$ is the dual graph to $\gamma(L)$, the (painted) nerve graph corresponding to an FAL diagram of $L$. Crushtaceans were originally defined relative to the FAL polyhedral decomposition into two identical totally geodesic hyperbolic polyhedra whose boundary faces are checkerboard colored. Since we will not make explicit use of this polyhedral decomposition, we refer the reader to \cite[Section 7]{ChDeWi} for details on the original definition.

\begin{figure}[ht]
	\centering
	\begin{overpic}[scale=0.60]{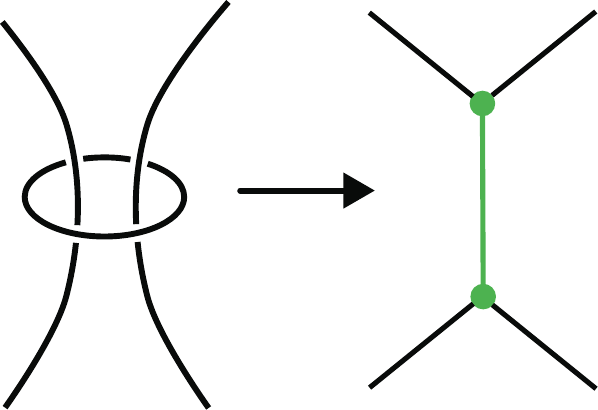}
		\put(33, 50){$L$}
		\put(90,50){$\Gamma$}
		\put(84,35){$p$}
	\end{overpic}
	\caption{The local picture of building a crushtacean $\Gamma$ from a flat FAL diagram for $L$.}
	\label{fig:BuildCrush}
\end{figure}

We would like to provide a graph-theoretic classification of crushtaceans that correspond to \textit{hyperbolic} flat FALs. We first note several of the essential characteristics of such graphs, which are also discussed in Zevenbergen \cite{Ze2021}.

\begin{lemma}
	\label{lem:crushprops}
Let $L$ be a (hyperbolic) flat FAL and let $\Gamma$ be a painted crushtacean graph corresponding to a flat FAL diagram for $L$. Then $\Gamma$ is a planar, cubic, $3$-connected graph with at least four vertices such that each vertex is incident to exactly one painted edge. 
\end{lemma}

\begin{proof}
	$\Gamma$ is clearly planar and cubic by construction. 	Figure \ref{fig:BuildCrush} also makes it clear that each vertex in $\Gamma$ is incident to exactly one painted edge. 

	 Next, we justify $\Gamma$ is $3$-connected. Purcell \cite[Lemma 2.3]{Pu4} showed that the nerve $\gamma$ of $L$ is a triangulation of $\mathbb{S}^{2}$. Thus, $\gamma$ is a polyhedral graph (the $1$-skeleton of a convex polyhedron). Since $\Gamma$ is dual to a polyhedral graph $\gamma$ it is also a polyhedral graph, and since all polyhedral graphs are $3$-connected, we conclude $\Gamma$ is  $3$-connected. 
	 
	 
	  $\Gamma$ is required to have at least four vertices since any hyperbolic FAL needs at least two crossing circles and each crossing circle contributes two distinct vertices in the construction of $\Gamma$. 
\end{proof}

\begin{lemma}
	\label{lem:simplecriteria}
	 If $\Gamma$ is a cubic and $3$-connected graph with at least four vertices, then it is a simple graph. In particular, a crushtacean is a simple graph. 
\end{lemma}

\begin{proof}
	Suppose $\Gamma$ is cubic and $3$-connected with at least four vertices. If $\Gamma$ had a self-loop $l$ based on a vertex $v$, then  $l$ and only one other edge, $e$, are incident to $v$. Then there are two more edges incident to  the other edge incident to $e$ which we label $v'$, and neither of these edges are incident to $v$. Thus, removing $v'$ would result in a disconnected graph that is non-trivial since $\Gamma$ has at least four vertices, which violates the fact that $\Gamma$ is $3$-connected. Next, suppose $\Gamma$ has at least two distinct edges incident to distinct vertices $v_1$ and $v_2$. If there is a third edge incident both $v_1$ and $v_2$, then since this graph is cubic, no other edges are adjacent to $v_1$ or $v_2$. In this case,  $\Gamma$ is either disconnected or $\Gamma$ has only two vertices, both of which are contradictions. So, instead, suppose $e_1$ is the one other edge incident to $v_1$ and incident to $v_{1}'$ with $v_{1}' \neq v_{2}$, and $e_2$ is the one other edge incident to $v_2$ and incident to $v_{2}'$, with $v_{2}' \neq v_{1}$. Then removing $v_{1}$ and $v_{2}'$  disconnects $\Gamma$ into two non-trivial components, which violates the fact that $\Gamma$ is $3$-connected. 
	
	The conclusion about crushtaceans immediately follows from Lemma \ref{lem:crushprops}.
\end{proof}

A simple planar graph $G$ is a \textbf{triangulation} of $\mathbb{S}^{2}$ if it admits a planar embedding where each face borders three edges and no distinct faces share more than one edge. 

\begin{prop}
	\label{prop:crushchar}
	A graph $\Gamma$ with at least four vertices represents a painted crushtacean for a hyperbolic flat FAL $L$ if and only if $\Gamma$ has the following properties:
	\begin{enumerate}
		\item planar,
		\item cubic,
		\item $3$-connected, and
		\item the set of painted edges is a perfect matching of $\Gamma$.
	\end{enumerate}
\end{prop}

\begin{proof}
Lemma \ref{lem:crushprops} takes care of the forward direction. For the backwards direction, suppose $\Gamma$ is a planar, cubic, $3$-connected graph with at least four vertices and where  the set of painted edges is a perfect matching of $\Gamma$. Let $\gamma$ be the dual graph to $\Gamma$. Our goal is to employ Lemma 2.4 from \cite{Pu4} and show that $\gamma$  is a simple graph, $\gamma$ is a triangulation of $\mathbb{S}^{2}$, and each triangle of $\gamma$ meets exactly one painted edge, which then implies that $\gamma$ a painted nerve for a hyperbolic flat FAL. Since $\Gamma$ is $3$-connected, it has a unique embedding in $\mathbb{R}^{2}$ up to isotopy \cite{WH1932}, and so,  $\Gamma$ and $\gamma$ are unique duals to one another. Lemma \ref{lem:simplecriteria} shows that $\Gamma$ is simple, making $\Gamma$ a polyhedral graph (simple, $3$-connected, and planar). Since the dual of a polyhedral graph is also a polyhedral graph, we conclude that $\gamma$ is simple, $3$-connected, and planar.

We now justify that $\gamma$ is a triangulation of $\mathbb{S}^{2}$. The previous paragraph shows that $\gamma$ is simple and planar. Since $\Gamma$ is cubic, we can conclude that each face of our planar embedding of $\gamma$ borders exactly three edges. Since $\Gamma$ is simple,  distinct faces in the planar embedding of $\gamma$ share at most one edge. Otherwise, if a two distinct faces shared at least two edges on their border in the planar embedding of $\gamma$, then in the dual $\Gamma$, there would be a multi-edge, violating $\Gamma$ being simple.

Finally, since the set of painted edges is a perfect matching of $\Gamma$, we see that each vertex of $\Gamma$ is incident to exactly one painted edge.  Thus, each triangle of the dual graph $\gamma$  intersects exactly one painted edge of $\Gamma$ in the interior of an edge of $\gamma$, as needed.  
\end{proof}

 Recall that we assumed any flat FAL under consideration is hyperbolic, and so, any corresponding crushtacean will have the properties highlighted in Lemma \ref{lem:simplecriteria} and Proposition \ref{prop:crushchar}.


\section{Symmetries of flat FALs via automorphisms of crushtacean graphs}
\label{sec:symFALviaautocrush}

Belleman \cite{Belleman2023} noted that any element of $Aut_{p}(\Gamma)$ induces an element of $Sym^{+}(\Sth, L)$ via explicitly constructing $L$ within a small neighborhood of a polyhedron in $\mathbb{R}^{3}$ whose $1$-skeleton is isomorphic to $\Gamma$.  Her work is unpublished and only provides a sketch of these results. Here, we construct a detailed proof that  provides more explicit details of the relation between $Aut_{p}(\Gamma)$ and $Sym^{+}(\Sth, L)$.

We now need to introduce an important class of totally geodesic surfaces contained in every flat FAL complement.  Given a fixed flat FAL diagram $D(L)$, each crossing circle $C_i$ bounds at least one twice-punctured disk $D_i$ in the complement $\Sth \setminus L$ that is punctured by some knot circles. In $\Sth \setminus L$, any such \textbf{crossing disk} is a totally geodesic thrice-punctured sphere; see \cite[Lemma 2.1]{Pu4}. It is possible for a single crossing circle to admit multiple distinct crossing disks, and this feature is analyzed in \cite{MRSTZ2025} in the context of belted-sum decompositions of FALs.

Suppose $L$ is a flat FAL with a fixed choice of crossing circles and knot circles coming from a fixed flat FAL diagram. As noted in the introduciton, we say that a symmetry of $L$ is type-preserving if this symmetry maps the set of crossing circles to the set of crossing circles, and thus, also maps the set of knot circles to the set of knot circles. 

\begin{named}{Theorem \ref{thm:main1}}
		\label{thm:Belleman}
	Let $L$ be a flat FAL with a painted crushtacean $\Gamma$. Then there exists a  monomorphism $\phi: Aut_{p}(\Gamma) \rightarrow Sym^{+}(\Sth, L)$. Furthermore, the image of $\phi$ only contains type-preserving symmetries of $L$.
\end{named}

\begin{proof}
		Let $L$ be a flat FAL with a painted crushtacean $\Gamma$. By Proposition \ref{prop:crushchar}, the graph $\Gamma$ is planar and $3$-connected. Steinitz's Theorem \cite{St1928} and the proof of this theorem show that $\Gamma$ can be realized as the $1$-skeleton, $P_1$, of a convex polyhedron $P$ in $\mathbb{R}^{3}$ where the edges of $P_1$ are all tangent to the unit sphere $\mathbb{S}^{2}$ in $\mathbb{R}^{3}$; see Theorem 2.8.11 in \cite{MoTh2001}. Mani \cite{Mani1971} provides an extension of Steinitz's Theorem which shows that there exists an embedding of $\Gamma$ on $P$ in $\mathbb{R}^{3}$ where the group $Aut(\Gamma)$ is isomorphic to $Isom(\mathbb{R}^{3}, P)$, the set of isometries of $\mathbb{R}^{3}$ that preserve $P$ (see Figure \ref{fig:Thm1.1}$(a)$). Since edges of $P_1$ are tangent to $\mathbb{S}^2$, every element of $Isom(\mathbb{R}^{3}, P)$ preserves $\mathbb{S}^2$ as well as the origin.  For the purpose of relating $Aut_{p}(\Gamma)$ to $Sym^{+}(\Sth, L)$, it will be convenient to project $P$ radially onto the unit sphere $\mathbb{S}^2$ (see Figure \ref{fig:Thm1.1}$(b)$).  Since elements of $Isom(\mathbb{R}^{3}, P)$ preserve the origin, and hence radial projection, they also preserve the projected $P$.  Thus $Isom(\mathbb{R}^{3}, P)$ is the group of isometries of $\mathbb{R}^{3}$ that preserve $\mathbb{S}^{2}$ along with this projected $1$-skeleton on $\mathbb{S}^{2}$. Hence we still have an isomorphism between $Aut(\Gamma)$ and isometries of $\mathbb{R}^{3}$ that preserve the projected $P$. We abuse notation and use $P$ to now represent $\mathbb{S}^{2}$ with its projected $1$-skeleton, $P_1$, which is (graph) isomorphic to $\Gamma$. In addition, the painting on $\Gamma$ induces a painting on certain edges of $P_1$ such that $Aut_{p}(\Gamma)$ coincides with $Isom_{p}(\mathbb{R}^{3}, P)$, the group of isometries of $P$ in $\mathbb{R}^{3}$ that preserve it's $1$-skeleton and this painting. Moving forward, we just need to establish a monomorphism $\phi: Isom_{p}(\mathbb{R}^{3}, P) \rightarrow Sym^{+}(\Sth, L)$.

	We now describe how to construct $L$ within a small neighborhood of $P$ in $\mathbb{R}^{3}$ so that $L$ is fixed setwise for every $f \in Isom_{p}(\mathbb{R}^{3}, P)$.  We begin by describing the crossing circles of $L$, which will link the painted edges of the projected $P_1$.  For each painted edge $e$ of $P_1$, let $v$ be the point where the original edge $e$ was tangent to $\mathbb{S}^2$, and let $\mathbb{P}$ be the plane through $v$ and orthogonal to $e$.  Let $U_e$ be the (geometric) circle in $\mathbb{P}$ defined by the following three properties:  it must be centered on the ray from the origin through $v$, for each $e$ its radius must be the same sufficiently small fixed $\delta$, and finally $U_e$ must be orthogonal to $\mathbb{S}^2$. The set $\{U_e\}$, taken over all painted edges of $P_1$, will be the crossing circles of $L$. Note that elements of $Isom_p(\mathbb{R}^{3}, P)$ permute the $U_e$ since they permute their corresponding $v$'s and preserve directions and distances.  
	
	The knot circles of $L$ are constructed on $\mathbb{S}^2$, and consist of arcs parallel to painted edges of $P_1$ together with the portions of unpainted edges joining their endpoints.  More precisely, for each painted edge $e$ and sufficiently small $\epsilon$, let $N_{\epsilon}(e)$ be an $\epsilon$-neighborhood of $e$ on $\mathbb{S}^2$.  Two edges $e_1, e_2$ of $\partial N_{\epsilon}(e)$ run ``parallel" to $e$, each intersecting the two unpainted edges on their respective sides of $e$. The knot circles of $L$ are comprised of the edges $e_i$, together with the portions of unpainted edges connecting them (see Figure \ref{fig:Thm1.1}$(c)$).   Since elements of $Isom_p(\mathbb{R}^{3}, P)$ preserve painted (and, therefore, unpainted) edges, as well as distances, they preserve the knot circles of $L$ as well.  Thus, every element of $Isom_p(\mathbb{R}^{3}, P)$ is a type-preserving isometry of $\mathbb{R}^{3}$ that maps $L$ to $L$. Before moving on, note that $L$ is a fixed set under inversion in $\mathbb{S}^2$ as well, since knot circles are fixed point-wise and crossing circles are orthogonal to the unit sphere.

\begin{figure}[ht]
\[\begin{array}{ccc}
\includegraphics{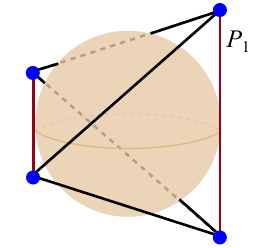} & \includegraphics{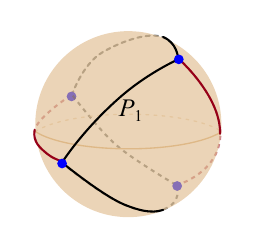} & \includegraphics{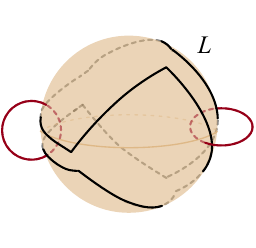}\\
(a)\textrm{ Mani's symmetric }P_1 & (b)\textrm{ } P_1 \textrm{ projected on } \mathbb{S}^2 & (c)\textrm{ Associated FAL }L
\end{array}
\]
	\caption{Geometrically realizing a painted crushtacean $P_1$ and its corresponding flat FAL $L$, so that $Isom_p(\mathbb{R}^3,P_1) \cong Aut_{p}(P_1)$ and each element of $Isom_p(\mathbb{R}^3,P_1)$ induces a symmetry of $L$.}
	\label{fig:Thm1.1}
\end{figure}

	 We now define the map $\phi: Isom_{p}(\mathbb{R}^{3}, P) \rightarrow Sym^{+}(\Sth, L)$. Given  $f \in Isom_{p}(\mathbb{R}^{3}, P)$, let $f^{\ast}:\mathbb{S}^3\to\mathbb{S}^3 $ be the extension of $f$ from $\mathbb{R}^3$ to $\mathbb{S}^3 = \mathbb{R}^{3} \cup \{ \infty\}$ that fixes $\infty$ and agrees with $f$ on $\mathbb{R}^3$. Then $f^{\ast}$ represents a  symmetry of $(\Sth, L)$ that is type-preserving relative to the flat FAL diagram where $P$ is the projection plane. With this in mind, we define $\phi$ as follows. If $f \in Isom_{p}(\mathbb{R}^{3}, P)$ is orientation-preserving, define $\phi(f) = f^{\ast}$. If $f \in Isom_{p}(\mathbb{R}^{3}, P)$ is orientation-reversing, define $\phi(f) \in Sym^+(\mathbb{S}^3,L)$ as the composition $r \circ f^{\ast}$, where $r$ is inversion in the unit $2$-sphere $P$. Since the composition of two orientation-reversing maps is orientation-preserving, we have that $\phi(f) \in Sym^{+}(\Sth, L)$ for this case. In addition, since $f$ and $r$ are both type-preserving on $L$, we also see that $f^{\ast}$ will be type-preserving. 
	 
	 We now show that the map $\phi: Isom_{p}(\mathbb{R}^{3}, P) \rightarrow Sym^{+}(\Sth, L)$ is a well-defined homomorphism. First, we claim that any extension $f^{\ast}$ commutes with $r$. Assuming this commutativity, the reader can easily verify that $\phi( f \circ g) = \phi(f) \circ \phi(g)$ for all possible cases. To prove our claim, let $x\in\mathbb{R}^3$ be any point other than the origin $\textbf{o}$.  Then its reflection across the unit sphere, $r(x)$, is the unique point on ray $\overrightarrow{\textbf{o}x}$ satisfying $\|r(x)\|\cdot\|x\| = 1$.  Since $f^{\ast}$ fixes the origin while preserving lines and distances, we see that $f^{\ast}(x)$ and $f^{\ast}\left(r({x})\right)$ satisfy the same properties, so that $f^{\ast}\left(r(x)\right) =  r \left(f^{\ast}(x)\right)$ for all points of $\mathbb{R}^3$ other than the origin.  Now $r$ always swaps the points $\textbf{0}$ and $\infty$ while $f^{\ast}$ fixes them, so composing them in either order always swaps $\textbf{0}$ and $\infty$. Thus $f^{\ast} \circ r$ and $r \circ f^{\ast}$ agree on $\mathbb{R}^3\cup\infty$, finishing the proof that $\phi$ is a homomorphism.
	 	 
	  To show $\phi$ is one-to-one, suppose $f \in Isom_{p}(\mathbb{R}^{3}, P)$ such that $\phi(f) = I$. Since $\phi(f) = f^{\ast}$ or $\phi(f) = r \circ f^{\ast}$, and $r$ acts as the identity on $P$, we see that $f^{\ast}$ must act the identity on $P$. However, if $f^{\ast}$ acts as the identity on $P$, then the same holds for $f$. Thus, $ker(\phi)$ is trivial and $\phi$ is one-to-one. 
	\end{proof}

Combining Theorem \ref{thm:main1} with previous work from the authors provides several useful results about $Sym^{+}(\Sth, L)$ and $Sym^{+}(\Sth \setminus L)$, which we now highlight.

Let $L$ be a flat FAL with reflection surface $R$ and fix a  crossing disk $D_i$ for each crossing circle $C_i$.  Further, let $u_i,v_i$ denote the points where $D_i$ intersects the knot circles of $L$, and let $e_i$ denote the edge of intersection in $D_i\cap R$ with endpoints $u_i,v_i$.  The \textbf{pre-crushtacean} $\Gamma'$ determined by the choice $D_i$ is the graph with vertices $u_i,v_i$ and edges given by the $e_i$ together with arcs of knot circles joining vertices.

Contracting the edges $e_i$ in $\Gamma'$, then splitting the vertices $u_i$ results in the crushtacean $\Gamma$ determined by the $D_i$.

\begin{lemma}
\label{lem:induceAuto}
	Let $L$ be a flat FAL whose complement admits a unique reflection surface $R\subset \Sth \setminus L$.  Choose one crossing disk $D_i$ for each crossing circle $C_i$ of $L$, and denote the resulting pre-crushtacean and crushtacean by $\Gamma'$ and $\Gamma$, respectively. If $f\in Sym^+(\Sth, L)$ permutes crossing disks, then $f$ induces an automorphism of $\Gamma$. In particular, if $L$ is $b$-prime and $f\in Sym^+(\Sth, L)$ permutes crossing circles, then $f$ induces an automorphism of $\Gamma$.
\end{lemma}

\begin{proof}
	Suppose $f\in Sym^+(\Sth, L)$ permutes the chosen set of crossing disks $\{D_i\}$, and so, $f$ also permutes the set of boundary crossing circles $\{C_i\}$ as well.  Let $\rho_f: \Sth \setminus L \to \Sth \setminus L$ be the hyperbolic isometry induced by $f$.  Since $L$ has a unique reflection surface, Corollary 3.13 of \cite{MiTr2023} implies $\rho_f(R) = R$.  Moreover, since $f$ permutes crossing circles, it must also permute knot circles of $L$. 
	
	Now let $\Gamma'$ denote the pre-crushtacean determined by the choice of crossing disks $\{D_i\}$.  Since $f$ permutes knot circles as well as this set of crossing disks, we have that $f(\Gamma') = \Gamma'$.  This implies that $f(\Gamma) = \Gamma$ and, since $f$ clearly induces an isomorphism between $\Gamma$ and $f(\Gamma)$, it induces an automorphism of $\Gamma$.
	
	If we suppose $L$ is $b$-prime, then Corollary 5.3 of \cite{MRSTZ2025} implies each crossing circle bounds a unique crossing disk.  Thus, preserving crossing circles implies preserving the crossing disks in this case.
\end{proof}

\begin{cor}
	\label{cor:bprimeunique}
	Let $L$ be a $b$-prime flat FAL whose complement admits a unique reflection surface. Then there exists a unique crushtacean for $L$, up isotopy in $\mathbb{R}^{2}$.
\end{cor}

\begin{proof}
	Since $L$ admits a unique reflection surface, there is a unique partition of the components of $L$ into crossing circles and knot circles. Since $L$ is $b$-prime, for each crossing circle, there exists a unique crossing disk in $\Sth \setminus L$. The uniqueness of knot circles and crossing disks determines a unique pre-crushtacean for $L$, and so, also determines a unique crushtacean for $L$ (up to graph automorphism). Since a crushtacaean is a planar $3$-connected graph, its embedding in $\mathbb{R}^{2}$ is unique up to isotopy by Whitney's Theorem \cite{WH1932}.
\end{proof}

We now can prove Corollary \ref{cor:main1} from the introduction.

\begin{named}{Corollary \ref{cor:main1}}
	Let $L$ be a $b$-prime flat FAL whose complement admits a unique reflection surface.  Then $Sym^+(\Sth, L) \cong Aut_{p}(\Gamma)$. If in addition $L$ is not a signature link, then $Sym^+(\Sth \setminus L) \cong Sym^+(\Sth, L) \cong Aut_{p}(\Gamma)$.
\end{named}

\begin{proof}
	Let $D(L)$ be the flat FAL diagram that corresponds to the crushtacean $\Gamma$. Theorem \ref{thm:main1} tells us that $\phi: Aut_{p}(\Gamma) \rightarrow Sym^{+}(\Sth, L)$ is a monomorphism. Since $L$ is $b$-prime and $\Sth \setminus L$ has a unique reflection surface, Lemma \ref{lem:induceAuto} shows that any type-preserving symmetry of $L$ is induced by an element of $Aut_{p}(\Gamma)$. Thus, it remains to show that $Sym^{+}(\Sth, L)$ does not contain any type-changing symmetries. If $Sym^{+}(\Sth, L)$ did contain a type-changing symmetry, then there would exist a flat FAL diagram for $L$ with a different partition of the components of $L$ into knot circles and crossing circles than that of $D(L)$. However, this would imply that $\Sth \setminus L$ contains two distinct reflection surfaces, which is a contradiction. 
	
	If in addition $L$ is not a signature link, then Theorem 1.3 of \cite{MiTr2023} implies $Sym^+(\Sth \setminus L)\cong Sym^+(\Sth, L)$ under our hypotheses.
\end{proof}

Both hypotheses in Corollary \ref{cor:main1} are necessities for its conclusion to hold. For instance, in Section \ref{sec:MultRef}, we  examine an infinite class of flat FALs $\{P_n\}$ where each $P_n$ is b-prime, but each $\Sth \setminus P_n$ admits multiple reflection surfaces and $[Sym^{+}(\Sth, P_n) : \phi(Aut_{p}(\Gamma(P_n)))] =2$. In that section, we also examine an infinite class of flat FALs $\{O_n\}$ where each $O_n$ is not b-prime, each $\Sth \setminus O_n$ admits multiple reflection surfaces, and $[Sym^{+}(\Sth, O_n) : \phi(Aut_{p}(\Gamma(O_n)))] =2n$. Next, in Figure \ref{fig:bComposite} we give an example of a b-composite FAL whose complement admits a unique reflection surface but $Sym^+(\Sth, L) \not\cong Aut_{p}(\Gamma)$. In particular, the reader can easily convince themselves that $Aut_{p}(\Gamma)\cong \mathbb{Z}_2$, while SnapPy calculates that $Sym(\Sth, L) \cong D_4 \times \mathbb{Z}_2$, where $D_4$ is the dihedral group of order eight.  Furthermore, the work of \cite[Theorem 1.3]{MiTr2023} shows that if $L$ is a signature link, that $\Sth \setminus L$ does admit symmetries (full-swaps) that do not correspond with symmetries of $L$. Thus, the hypothesis that $L$ is not a signature link is a necessity for $Sym^{+}(\Sth \setminus L) \cong Sym^{+}(\Sth, L)$.

\begin{figure}[ht]
	\centering
	\begin{overpic}{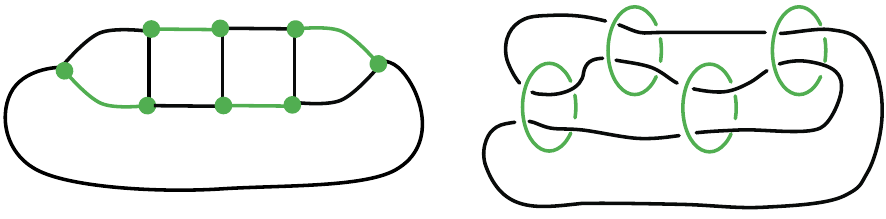}
		\put(5,20){$\Gamma$}
		\put(54,20){$L$}
	\end{overpic}

	\caption{A painted crushtacean $\Gamma$ and its $b$-composite flat FAL $L$ for which $Sym^+(\Sth, L) \not\cong Aut_{p}(\Gamma)$.}
	\label{fig:bComposite}
\end{figure}


\section{Symmetry Groups of $b$-prime flat FALs with a unique reflection surface}
\label{sec:classifying}

The previous section motivates a classification of automorphism groups of painted crushtaceans as an important tool for studying symmetry groups of flat FALs. To start, we will take advantage of some important properties of $3$-connected planar graphs. The following theorem was noted in \cite[Section 6]{KlNeZe2022}, and is an application of Whitney's Theorem \cite{WH1932} and Mani's Theorem \cite{Mani1971}. Recall that $O(3)$ designates the orthogonal group of dimension $3$, which is precisely the group of isometries of the $2$-sphere. 

\begin{thm}
	\label{thm:O(3)isometries}
	Let $X$ be a $3$-connected planar graph. Then $Aut(X)$ is isomorphic to a finite subgroup of isometries of $\mathbb{S}^{2}$, i.e., $Aut(X) \cong G$ where $G \leq O(3)$ and $|G| < \infty$. 
\end{thm} 

Corollary \ref{cor:main1} tells us that if $L$ is a $b$-prime flat FAL whose complement admits a unique reflection surface, then $Sym^{+}(\Sth, L) \cong Aut_{p}(\Gamma(L))$. Under these hypotheses,  Theorem \ref{thm:O(3)isometries} implies that $Sym^{+}(\Sth, L)$ is isomorphic to a finite subgroup of $O(3)$. This motivates the main goal for the rest of this section: show that for every finite subgroup $G$ of $O(3)$, there exists a $b$-prime flat FAL $L$ such that $Sym^{+}(\Sth, L) \cong G$. First, we review the classification of finite subgroups of $O(3)$, following the approach given in \cite{Senechal1990}. 


\subsection{Finite subgroups of $O(3)$}
\label{subsec:finitesubO3}

Recall that $SO(3)$ is the special orthogonal group of dimension $3$, which can also be described as the group of orientation-preserving isometries of the $2$-sphere. All elements of $SO(3)$ can be described as rotations about an axis through the center of the $2$-sphere. We immediately have the group decomposition $O(3) = SO(3) \times \{ \pm I \}$.  Here, we use $D_n$ to denote a dihedral group of order $2n$ for $n \geq 3$, $A_n$ to denote the alternating group on $n$ letters, and $S_n$ to denote the symmetric group on $n$ letters. 

If $G \leq SO(3)$ is finite, then most cases are covered by the following five possibilities which are highlighted as a pair corresponding to an abstract group and the polyhedron whose group of orientation-preserving symmetries is isomorphic to that group:
\begin{enumerate}
	\item ($\mathbb{Z}_n$, $n$-gonal pyramid with $n \geq 4$)
	\item ($D_n$, $n$-gonal prism or $n$-gonal antiprism with $n \geq 3$ and $n \neq 4$)
	\item ($A_4$, regular tetrahedron)
	\item ($S_4$, cube)
	\item ($A_5$, dodecahedron)
\end{enumerate}

In addition, we could also have the trivial group, $\mathbb{Z}_2$, $\mathbb{Z}_3$, and $D_4$. 

Now suppose $G \leq O(3)$ is finite where $G$ contains an orientation-reversing symmetry.  These possibilities are highlighted below as a pair corresponding to an abstract group and the polyhedron whose full group of symmetries is isomorphic to that group (when such a description exists). The pyritohedron is an irregular dodecahedron whose faces are all identical irregular pentagons.  The rhombic disphenoid is an irregular tetrahedron with four identical scalene triangles as faces.

\begin{enumerate}
	\item ($\mathbb{Z}_{2n}$ for $n$ odd and $\mathbb{Z}_{n} \times \mathbb{Z}_{2}$ for $n$ even, -)
	\item ($D_n \times \mathbb{Z}_{2}$ for $n$ even and $D_{2n}$ for $n$ odd, $n$-gonal prism)
	\item ($D_n \times \mathbb{Z}_{2}$ for $n$ odd and $D_{2n}$ for $n$ even, $n$-gonal antiprism) 
	\item ($D_{n}$, $n$-gonal pyramid)
	\item ($A_4 \times \mathbb{Z}_2$, pyritohedron)
	\item ($S_4 \times \mathbb{Z}_2$, cube)
	\item ($A_5 \times \mathbb{Z}_2$, dodecahedron)
	\item ($\mathbb{Z}_{2} \times \mathbb{Z}_{2}$, rhombic disphenoid)
	\item ($S_4 = A_4 \rtimes \mathbb{Z}_2$, tetrahedron)
\end{enumerate}


\subsection{Building crushtaceans with prescribed automorphism groups}
\label{subsec:buildingcrush}

Our goal is to show that given any $G \leq O(3)$ with $|G|< \infty$, there exists a painted crushtacean $\Gamma$ such that $Aut_{p}(\Gamma) \cong G$. Furthermore, we will build these painted crushtaceans so that their corresponding flat FALs are always $b$-prime and their corresponding  complements have a unique reflection surface. 

To start, we prove a useful lemma that shows how to build a painted crushtacean for a $b$-prime flat FAL with a prescribed group of automorphisms. Let $\Gamma'$ be a connected, simple graph. A \textbf{cycle replacement} on a vertex $v \in V(\Gamma')$ with $val(v)=n \geq 3$ is the process of replacing $v$ with an $n$-cycle. Now suppose for all $v \in V(\Gamma')$, we have $val(v) \geq 3$. Then a \textbf{cycle expansion} of $\Gamma'$ is the graph $\Gamma$ resulting from performing a cycle replacement on every vertex of $\Gamma'$. See Figure \ref{fig:CycleExpansion} for an example of a cycle expansion of a graph.

\begin{figure}[ht]
	\centering
	\begin{overpic}[scale=0.70]{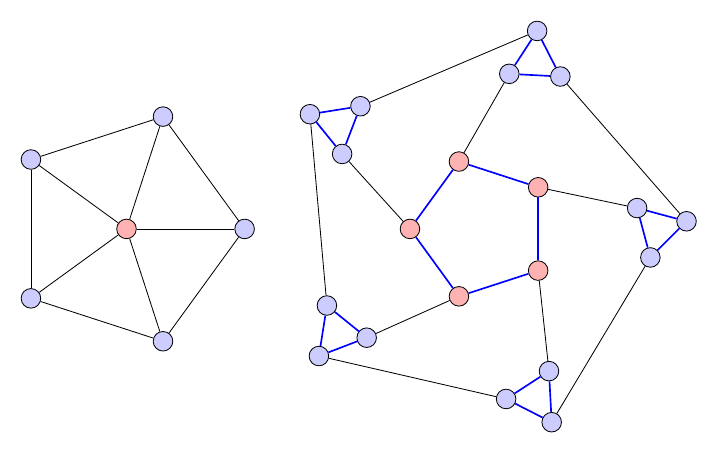}
		\put(13,50){$\Gamma' = W_5$}
		\put(60,58){$\Gamma$}
	\end{overpic}
	\caption{The wheel graph $\Gamma' = W_5$ is depicted on the left. The cycle expansion of $\Gamma'$, labeled as $\Gamma$, is depicted on the right with original edges in black and new edges in blue.}
	\label{fig:CycleExpansion}
\end{figure}

For our purposes, the graph $\Gamma'$ will be simple, planar, and $3$-connected.  A cycle replacement on a vertex of a planar graph can be achieved in the plane, so the cycle expansion $\Gamma$ of $\Gamma'$ is planar as well.  We introduce some terminology and make some observations relating the planar embeddings of $\Gamma$ and $\Gamma'$.  It makes sense to say ``the" planar embeddings since we will see, in Lemma \ref{lem:CycleExpansion} below, that both $\Gamma$ and $\Gamma'$ have unique embeddings in $\mathbb{S}^2$ up to isotopy.


 	 We call a cycle in $\Gamma$ a \textbf{new $n$-cycle} if it was created via cycle expansion on a vertex of degree $n$ in $\Gamma'$. An edge $e \in E(\Gamma)$ is a \textbf{new} edge if it is an edge in a new $n$-cycle; otherwise, $e$ is an \textbf{original} edge. A region $R$ of the planar embedding of $\Gamma$ is \textbf{new} if its boundary is a new $n$-cycle of $\Gamma$; otherwise, $R$ is referred to as an \textbf{original} region.  The definition of a cycle expansion on $\Gamma'$ induces a fair amount of structure on the adjacency of edges and regions in $\Gamma$.

We discuss the lengths of boundary cycles for both original and new regions. Observe that the boundary of an original region in $\Gamma$ has an even number of at least $6$ edges.  Indeed, the boundary cycle $b$ of an original region $R$ in $\Gamma$ has an even number of edges since it is obtained from the boundary $b'$ of its corresponding region $R'$ in $\Gamma'$ by introducing a new edge between each pair of adjacent edges of $b'$.  This observation also implies that the edges of $b$ alternate between original and new as one traverses the cycle.  Moreover, since $\Gamma'$ is simple, each boundary cycle of $\Gamma'$ has at least $3$ edges.  Hence the boundary of each original region in $\Gamma$ has an even number of at least $6$ edges.  Also note that the length of a boundary cycle for a new region in $\Gamma$ equals the degree of the vertex $v'$ used to create it.

We now discuss the types of regions adjacent to, and connected by both original and new edges in $\Gamma$. Each original edge $e\in E(\Gamma)$ is adjacent to original regions on both sides; whereas, if $e$ is new it bounds a new region on one side and an original one on the other. We say that two regions with disjoint boundaries are \textbf{connected} by the edge $e$ if one endpoint of $e$ lies in each boundary cycle.  With this definition, note that an original edge connects two new regions while a new edge connects two original regions.

\begin{lemma}
	\label{lem:CycleExpansion}
	Let $\Gamma'$ be a simple, planar, $3$-connected graph. Perform a cycle expansion on $\Gamma'$ to obtain $\Gamma$ and paint all the original edges of $\Gamma$. Then $\Gamma$ is a painted crushtacean for a $b$-prime flat FAL and $Aut(\Gamma) = Aut_{p}(\Gamma) \cong Aut(\Gamma')$. 
\end{lemma}

\begin{proof}
	Since $\Gamma'$ is $3$-connected, we have $val(v) \geq 3$ for all $v \in \Gamma'$, and so, a cycle expansion on $\Gamma'$ is well-defined. 
	
	We first show that $\Gamma$ is a painted crushtacean for a flat FAL. By construction, since $\Gamma'$ is simple and planar, the graph $\Gamma$ will also be simple and planar. In addition, each vertex of $\Gamma$ is incident to one original edge and two distinct new edges coming from the same new $n$-cycle. Thus, $\Gamma$ is cubic. Furthermore, since only original edges of $\Gamma$ are painted, we see that each vertex is incident to exactly one painted edge.	To show $\Gamma$ is $3$-connected, suppose not, that is, suppose there exists a pair of vertices $\{v_1, v_2 \}$ in $\Gamma$ such that their removal disconnects $\Gamma$. Note that, every vertex of $\Gamma$ lies on a new $n$-cycle and every such vertex is adjacent to only one original edge; call such an edge $e_i$ for $v_i$.  Let $e_{i}' \in E(\Gamma')$ be the corresponding edge for $\Gamma'$. Now, remove  $\{ v_{1}', v_{2}'\}$ from $\Gamma'$ where $v_{i}'$ is incident to $e_{i}'$  in $\Gamma'$  and $v_{i}'$ was replaced by the $n$-cycle that contains $v_{i}$. This will disconnect $\Gamma'$, which contradicts the assumption that $\Gamma'$ is 3-connected. Thus, $\Gamma$ is a painted crushtacean for a flat FAL $L$. 
	
	Next, we show that every $3$-edge cut of $\Gamma$ is thrice-painted, which in turn justifies that $L$ is $b$-prime by Theorem 5.2 of \cite{MRSTZ2025}. Any $3$-edge cut of $\Gamma$ either includes one painted edge or three painted edges; see the proof  of \cite[Theorem 5.2]{MRSTZ2025}. So, suppose there exists a  $3$-edge cut on  $\Gamma$ that has only one painted edge, i.e., this $3$-edge cut involves two edges, $e_1$ and $e_2$, on new cycles and one original edge $e$, with $e'$ standing for the corresponding edge in $\Gamma'$. First, suppose $e_1$ and $e_2$ lie on the same new cycle. Let $v \in V(\Gamma')$ such that the cycle replacement of $v$ created the new cycle that contains $e_1$ and $e_2$. Let $v'$ be a vertex incident to $e'$ in $\Gamma'$. Then we have that $\Gamma' - \{v, v'\}$ is disconnected, which violates the fact that $\Gamma'$ is $3$-connected. So, we now assume that $e_1$ and $e_2$ lie on distinct new cycles in $\Gamma$. If you cut along a single $e_i$ within a new cycle $C$, then any vertex in that new cycle is still path-connected in $C - e_i$ to any other vertex within that $n$-cycle. However, this would imply that $\Gamma' - e'$ is disconnected (or similarly, $\Gamma' - v$ where $v$ is incident to $e$), which violates the fact that $\Gamma'$ is $3$-connected. 
	
	
We now justify that $Aut_{p}(\Gamma) \cong Aut(\Gamma')$.  Let $E_p$ represent the set of painted edges in $\Gamma$. Then any $f \in Aut_{p}(\Gamma)$ must map $E_p$ to $E_p$. Since $\Gamma - E_p$, is the collection of disjoint new cycles, $f$ must map a new $n$-cycle to some other new $n$-cycle. By replacing each new $n$-cycle with a degree $n$ vertex, we see that $f$ clearly corresponds with an $f' \in Aut(\Gamma')$. At the same time, any $g' \in Aut(\Gamma')$ corresponds with a $g \in Aut_{p}(\Gamma)$ between a cycle expansion on $\Gamma'$ and a cycle expansion on $g'(\Gamma')$.

It remains to show $Aut(\Gamma) = Aut_{p}(\Gamma)$.  More precisely, since $Aut_{p}(\Gamma) \subseteq Aut(\Gamma)$ it's enough to show $Aut(\Gamma) \subseteq Aut_{p}(\Gamma)$. First, suppose $N$ and $O$ are a new and original region of $\Gamma$, respectively, and that $\varphi\in Aut(\Gamma)$ is such that $\varphi(N) = O$.  We will show this leads to a contradiction, which will then quickly get us to the desired result. We claim that $\varphi$ maps every new region of $\Gamma$ to an original one.  To see this let $e_1,\dots, e_j$ be the original edges of $\Gamma$ with one endpoint on $\partial N$, and let $N_1,\dots,N_j$ denote the new regions with $e_i$ connecting $N_i$ and $N$.  Since $\varphi(N) = O$ is an original region, each $\varphi(e_i)$ is a new edge which must connect $O$ to an original region $O_i$.  Thus $\varphi(N_i) = O_i$ for all $1\le i \le j$, and $\varphi$ maps all new regions of $\Gamma$ connected with $N$ to original regions.  Consecutively repeating this argument with the $N_i$ and their neighboring new regions eventually terminates with showing $\varphi$ maps every new region to an original one. Since $\Gamma$ admits a unique planar embedding, we know that $v-e+f =2$, where $v = |V(\Gamma)|$, $e = |E(\Gamma)|$, and $f$ is the number of regions in this planar embedding. Since every edge is incident to two distinct vertices and $\Gamma$ is cubic, we have that $3v = 2e$. At the same time, every edge bounds exactly two distinct faces for the given embedding. Furthermore, every face must be bound by some $m$-cycle of $\Gamma$ where $m \geq 6$. This fact is true for any original region, and by the previous paragraph, it must also be true for any new region. Thus, $2e \geq 6f$, i.e., $\frac{e}{3} \geq f$. Substituting into the Euler characteristic equation gives $\frac{2e}{3} - e + \frac{e}{3} \geq 2$, which implies $0 \geq 2$, which is a contradiction.  Thus, any  $\varphi \in Aut(\Gamma)$ must map new regions to new regions, and so, must also map original regions to original regions. As a result, $\varphi$ must map painted edges to painted edges, showing $Aut(\Gamma) \subseteq Aut_{p}(\Gamma)$, as needed. 
\end{proof}

\begin{thm}
	\label{thm:bprimeSymClass}
For every $G \leq O(3)$ with $|G|<\infty$, there exists an infinite collection of b-prime flat FALs $\{L_i \}$ such that $Sym^{+}(\Sth, L_i) \cong G$ for all $i$ and $L_i \not\simeq L_j$ for $i \neq j$. 
\end{thm}

\begin{proof}
Let $G \leq O(3)$ with $|G|<\infty$ be given.	By Main Corollary 8.12D from \cite{Babai1973}, there exists a $3$-connected planar simple graph $\Gamma'$ such that $Aut(\Gamma') \cong G$. Perform a cycle expansion on $\Gamma'$ to obtain $\Gamma_1$ and paint all the original edges of $\Gamma_1$. Then by Lemma \ref{lem:CycleExpansion}, $\Gamma_1$ is the painted crushtacean for a $b$-prime flat FAL $L$ with $Aut(\Gamma_1)  \cong Aut(\Gamma')  \cong G$. Now, since $\Gamma_{1}$ is the crushtacean for a flat FAL, we also know $\Gamma_1$ is simple, planar, and $3$-connected. Thus, we can  perform a cycle expansion on $\Gamma_1$ and apply Lemma \ref{lem:CycleExpansion} to obtain a painted crushtacean $\Gamma_{2}$ where $Aut(\Gamma_2)  \cong Aut(\Gamma_1)  \cong G$. We can repeat this process to obtain a sequence of painted crushtacean graphs $\{\Gamma_{i}\}_{i=1}^{\infty}$ such that $Aut(\Gamma_i) \cong Aut(\Gamma_j) \cong G$, for all $i,j$. 

Let $L_i$ denote the flat FAL corresponding to painted crushtacean $\Gamma_i$. We claim that each $\Sth \setminus L_i$ admits a unique reflection surface. Section \ref{sec:MultRef} classifies  painted crushtaceans of flat FALs whose complements admit multiple reflection surfaces. For a $\Gamma_i$ obtained by cycle expansion and painted as described above, non-painted edges only occur as edges in sets of disjoint $n$-cycles. However, none of the painted crushtaceans described in Section \ref{sec:MultRef} have this property. Thus, each  $\Sth \setminus L_i$ admits a unique reflection surface, and so, Corollary \ref{cor:main1} implies that $Sym^{+}(\Sth, L_i) \cong G$, for each $i$. Furthermore, since $\Gamma_{i+1}$ has strictly more painted edges than $\Gamma_{i}$, we see that $L_{i+1}$ has strictly more crossing circles than $L_i$. Since a $b$-prime flat FAL with a unique reflection surface has a unique set of crossing circles, we can conclude $L_i \not\simeq L_j$ for $i \neq j$. 
\end{proof}

With a little more work one can ensure that the infinite family $\{L_i \}$ from Theorem \ref{thm:bprimeSymClass} contains no signature links, and further conclude that $Sym^{+}(\Sth\setminus L_i) \cong Sym^{+}(\Sth, L_i)$ by Corollary \ref{cor:main1}.  With this stronger result in mind, we investigate the relationship between cycle expansion and signature links. Signature links are a special class of flat FALs, and only one property of signature links is needed for our purposes (we refer the interested reader to \cite[Definition 5.1]{MiTr2023} for the complete definition of a signature link). The significant property in what follows is that a signature link contains a knot circle $K_f$ which is linked to every other knot circle by exactly one crossing circle.  In what follows, we determine a combinatorial criterion for when the flat FAL corresponding to a given crushtacean is not a signature link, which allows us to show in Lemma \ref{lem:MultCycEx} that the links $\{L_i\}$, for $i \ge 2$, are not signature links.  Corollary \ref{cor:InfSysComp2} contains the desired strengthening of Theorem \ref{thm:bprimeSymClass}.

\begin{lemma}
\label{lemma:CombSigLink}
Let $\Gamma'$ be a simple, planar, $3$-connected graph in the plane $\mathcal{P}$, let $\Gamma$ be the crushtacean resulting from a cycle expansion on $\Gamma'$ with all original edges painted, and let $L$ be the flat FAL associated with $\Gamma$.  If $L$ is a signature link, then at least one region of $\mathcal{P}\setminus\Gamma'$ is adjacent to all regions of $\mathcal{P}\setminus\Gamma'$.
\end{lemma}

\begin{proof}
To prove this combinatorial requirement for an FAL to be a signature link we require a slightly modified version of the construction of $L$ from $\Gamma$ in Theorem \ref{thm:main1}.  In particular, we will show that $L$ can be constructed so that its knot  circles are in one-to-one correspondence with the original regions of $\Gamma$ and that two knot circles are linked by a crossing circle if and only if their corresponding regions are adjacent in the plane.

Assume $\Gamma$ is the projection onto $\mathbb{S}^2$ of the symmetric polyhedron $P$ guaranteed by Mani's Theorem (as in Figure \ref{fig:Thm1.1}$(b)$), and construct $L$ as in the proof of Theorem \ref{thm:main1}.  Arcs parallel to painted edges are glued together by segments of unpainted edges to construct the knot circles of $L$. If $\Gamma$ is a cycle expansion, painted edges adjacent to the same unpainted edge are on the same side of that cycle expansion (see Figure \ref{fig:CombSigLink}$(a)$).  In this case, the knot circle $K$ of $L$ defined in Theorem \ref{thm:main1} stays on the same side of an unpainted edge $e$ and we can isotope it into the interior of the original region adjacent to $e$.  Note that this cannot be done when the adjacent painted edges are on opposite sides of $e$, as in Figure \ref{fig:Thm1.1}$(c)$. The result is that each original region $R$ contains one knot circle $K_R$ which is the boundary of an $\epsilon$-neighborhood of $\partial R$ (see Figure \ref{fig:CombSigLink}$(b)$).  Furthermore, observe that two knot circles are linked when their corresponding regions share a painted edge.  Since the original edges of $\Gamma$ are painted, this occurs precisely when their corresponding regions are adjacent.  Thus the knot circles of the flat FAL $L$ associated with $\Gamma$ have the desired properties.

Now suppose that $\Gamma'$, $\Gamma$ and $L$ satisfy the hypotheses, and that $L$ is a signature link.  Then $L$ has a knot circle $K_f$ which, among other things, is linked to all other knot circles of $L$ by a crossing circle.  By the above arguments, this implies that the original region $R_f$ corresponding to $K_f$ must be adjacent to all other original regions, proving the result.
\end{proof}

\begin{figure}[ht]
\begin{center}
	\includegraphics{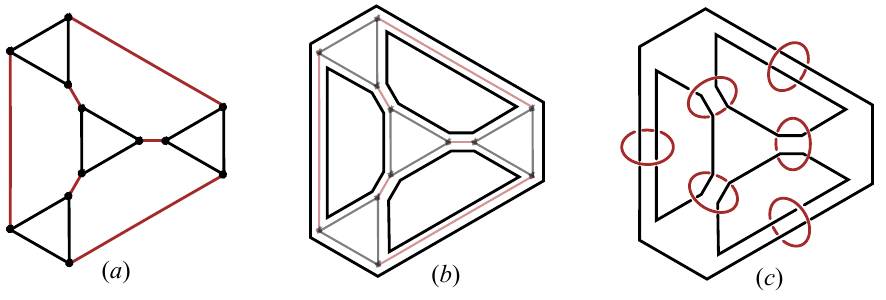}
\end{center}
	\caption{Geometrically building the flat FAL on a cycle expansion crushtacean.}
		\label{fig:CombSigLink}
\end{figure}

We remark that cycle expansion can lead to signature links.  For example, the cycle expansion $\Gamma$ depicted in Figure \ref{fig:CycleExpansion} gives rise to the flat FAL of Figure \ref{fig:FlatFAL}, which is a signature link.  However, we can use the contrapositive of Lemma \ref{lemma:CombSigLink} to show that performing more than one cycle expansion on a graph never leads to a signature link.  As in the proof of Theorem \ref{thm:bprimeSymClass}, we let $\left\{ \Gamma_i\right\} = \left\{\Gamma_1, \Gamma_2, \Gamma_3,\dots\right\}$ be the sequence of crushtaceans in which $\Gamma_1$ is the cycle expansion of $\Gamma'$ and, for $i > 1$, each $\Gamma_{i}$ is the cycle expansion of its predecessor.

\begin{lemma}
\label{lem:MultCycEx}
Let $\Gamma'$ be a simple, planar, 3-connected graph, let $\left\{ \Gamma_i\right\}$ be the sequence of crushtaceans just described, and let $L_i$ be the flat FAL corresponding to $\Gamma_i$.  Then for $i \ge 2$, the link $L_i$ is not a signature link.
\end{lemma}

\begin{proof}
To proceed, we let $\Gamma$ be a crushtacean resulting from cycle expansion on the simple, planar, 3-connected graph $\Gamma_0$ and show that no region of $\mathcal{P}\setminus\Gamma$ can be adjacent to all other regions.  Repeatedly applying  Lemma \ref{lemma:CombSigLink} to iterated cycle expansions of $\Gamma$ will complete the proof.

Let $R$ be an original region of $\Gamma$ with corresponding region $R_0$ of $\Gamma_0$.  Then the new regions that $R$ is adjacent to are exactly those resulting from cycle replacements on vertices of the boundary cycle of $R_0$.  Since $\Gamma_0$ cannot be a cycle, it must have at least one vertex $v$ not in $\partial R_0$, and $R$ will not be adjacent to the new region $N_v$ resulting from cycle replacement on $v$.  Thus no original region of $\mathcal{P}\setminus \Gamma$ is adjacent with all others.

Now let $N_v$ be a new region of $\mathcal{P}\setminus\Gamma$ obtained by cycle replacement on $v\in V(\Gamma_0)$.  Then $N_v$ is adjacent to the original regions of $\Gamma$ that correspond to the regions of $\Gamma_0$ adjacent to $v$.  We claim that not all regions of $\Gamma_0$ are adjacent to $v$. Assuming our claim, we see that the new region $N_v$ is not adjacent to all regions of $\Gamma$.

We now justify our claim. Note that, for any  wheel $W_n$ with $n \geq 3$, we see that $\Gamma_0 = W_n$ satisfies our claim. Furthermore, the work of Tutte \cite{Tu1961} shows that any simple $3$-connected graph can be built from some wheel $W_n$ by applying a finite series of edge additions and vertex splits (which we collectively refer to as Tutte moves) that preserve being $3$-connected. It is easy to see that performing a series of Tutte moves on $W_n$ (in a way that preserves planarity) will still result in a graph with the desired property in our claim. Since any $3$-connected, simple, planar graph can be built in this manner, $\Gamma_0$ must be have the desired property.

Thus no $\Gamma_i$ has a region adjacent to all others, so Lemma \ref{lemma:CombSigLink} implies that $L_{i+1}$ is not a signature link for all $i\ge 1$. 
\end{proof}


\begin{cor}
\label{cor:InfSysComp2}
For every $G \leq O(3)$ with $|G|<\infty$, there exists an infinite collection of b-prime flat FALs $\{L_i \}$ such that $Sym^{+}(\Sth \setminus L_i) \cong Sym^{+}(\Sth, L_i) \cong G$ for all $i$ and $L_i \not\simeq L_j$ for $i \neq j$.
\end{cor}

\begin{proof}
One mimics the proof of Theorem \ref{thm:bprimeSymClass} to find $\Gamma'$ with $Aut(\Gamma') \cong G$. Then again following the proof of Theorem \ref{thm:bprimeSymClass}, construct the sequence of crushtaceans $\left\{ \Gamma_i \right\}$ with corresponding $b$-prime  flat FALs $\{L_i\}$, all of which admit a unique reflection surface.  If $L_1$ is a signature link, omit it; otherwise, use the complete sequence.  In either case, Lemma \ref{lem:MultCycEx} shows that none of the $L_i$ are signature links, and so, Corollary \ref{cor:main1} provides the additional isomorphism $Sym^{+}(\Sth, L_i) \cong Sym^{+}(\Sth, L_i)$.
\end{proof}

Our work gives an explicit construction for building $b$-prime flat FALs meeting the qualifications of Corollary \ref{cor:InfSysComp2}. One only needs a simple, planar $3$-connected graph $ \Gamma'$ with the appropriate automorphism group. From there, one can iteratively perform  cycle expansions, build the resulting crushtacean, and then build the resulting flat FALs. Finding a $\Gamma'$ that meets these qualifications without referring to Main Corollary 8.12D from \cite{Babai1973}  is frequently straightforward, and we provide a few examples to highlight this. 

\begin{ex}
If $\Gamma'$ is the $1$-skeleton of a regular tetrahedron, cube, regular dodecahedron, regular $n$-prism, or regular $n$-antiprism, then $\Gamma'$ is a polyhedral graph (simple, planar, $3$-connected), and the automorphisms of $\Gamma'$ are the same as the symmetries of its corresponding solid. Thus, such $\Gamma'$ allow one to construct infinite sets of distinct b-prime flat FALs whose orientation-preserving symmetry groups (both of the link and its corresponding complement) are isomorphic to either $S_4$, $S_4 \times \mathbb{Z}_{2}$, $A_5 \times \mathbb{Z}_{2}$, $D_n \times \mathbb{Z}_{2}$, or $D_{2n}$, respectively.
\end{ex}

\begin{ex}
For each $n \geq 4$, the wheel graph $W_n$ is simple, planar, $3$-connected, and $Aut(W_n) = D_n$. For a fixed $n$,  we could set $\Gamma' = W_n$ to build an infinite class of distinct $b$-prime flat FALs $\{L_i\}_{=2}^{\infty}$ such that $Sym^{+}(\Sth \setminus L_i) \cong Sym^{+}(\Sth, L_i) \cong D_n$, for all $i$. In this case,  $L_1$ is a signature link, and we have  $D_n \cong Sym^{+}(\Sth, L_1) \subsetneq Sym^{+}(\Sth \setminus L_1)$. Figure \ref{fig:CycleExpansion} provides a visual of $W_5$ and its first cycle expansion. A diagram for the resulting $b$-prime flat FAL $L_1$ is given in Figure \ref{fig:FlatFAL}. We also note that $W_n$ for $n \geq 4$ provides a planar representation for the regular $n$-gonal pyramid whose symmetry group is $D_n$.  
\end{ex}

\begin{ex}
One can see that the graph $\Gamma_{T}$ depicted in Figure \ref{fig:TetSym} is planar, simple, and $3$-connected with $Aut(\Gamma_{T}) = A_4$. Recall that $A_4$ is isomorphic to the group of orientation-preserving symmetries of the regular tetrahedron. Thus, one can set $ \Gamma' = \Gamma_{T}$ and apply cycle expansions as discussed in Corollary \ref{cor:InfSysComp2}. Furthermore, one could apply a similar graph-theoretic constructions to produce  planar, simple, $3$-connected graphs that  only have the orientation-preserving symmetries of the cube or the dodecahedron. Figure \ref{fig:TetSym} is a modification of the example from https://math.stackexchange.com/questions/2461895/a-graph-whose-automorphism-group-is-the-alternating-group. 
\end{ex}

\begin{figure}[ht]
	\centering
	\begin{overpic}[scale=0.70]{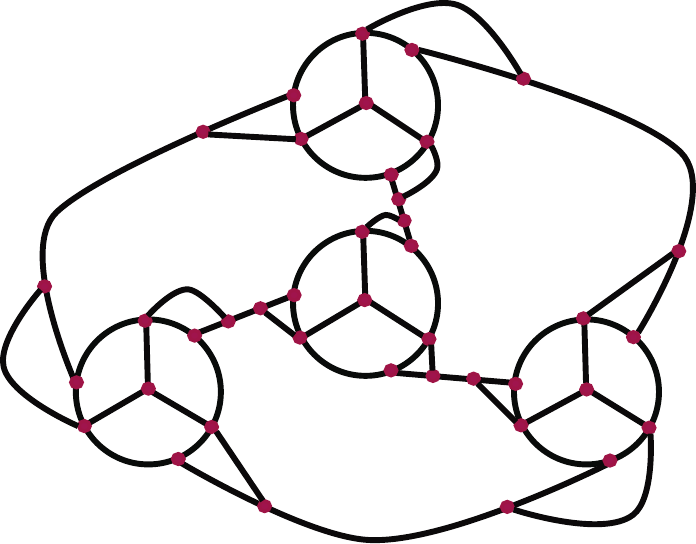}
		\end{overpic}
	\label{fig:TetSym}
	\caption{A $3$-connected, simple, planar graph $\Gamma_{T}$ where $Aut(\Gamma_{T}) = A_4$.}
\end{figure}


\section{Symmetry groups of flat FALs with multiple reflection surfaces}
\label{sec:MultRef}

A complete classification of flat FALs with multiple reflection surfaces  was given by the authors in previous work, which we state below. The links $P_n$ and $O_n$ are described in more detail in upcoming subsections, and Figure \ref{fig:BP4O4} provides diagrams of $P_4$ and $O_4$. 

\begin{thm}\cite[Theorem 1.2]{MiTr2023}
	\label{thm:MT2}
	Suppose $M = \mathbb{S}^{3} \setminus L$ is a flat FAL complement with multiple distinct reflection surfaces. Then either 
	\begin{itemize}
		\item $L$ is equivalent to the Borromean rings and $M$ contains exactly three distinct reflection surfaces, or
		\item $L$ is equivalent to $P_n$ with $n \geq 3$, or $O_n$ with $n \geq 2$, and $M$ contains exactly two distinct reflection surfaces. 
	\end{itemize}
\end{thm}

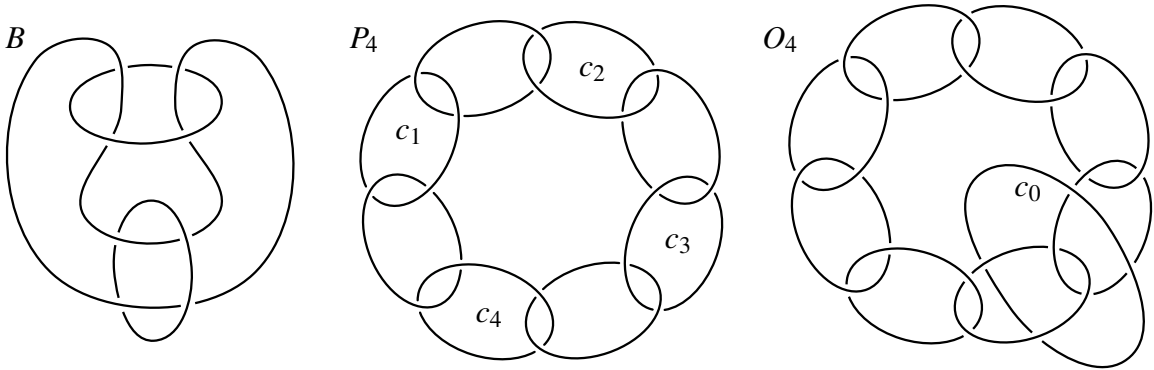
\begin{figure}[ht]
	\centering
	\begin{overpic}[width = \textwidth]{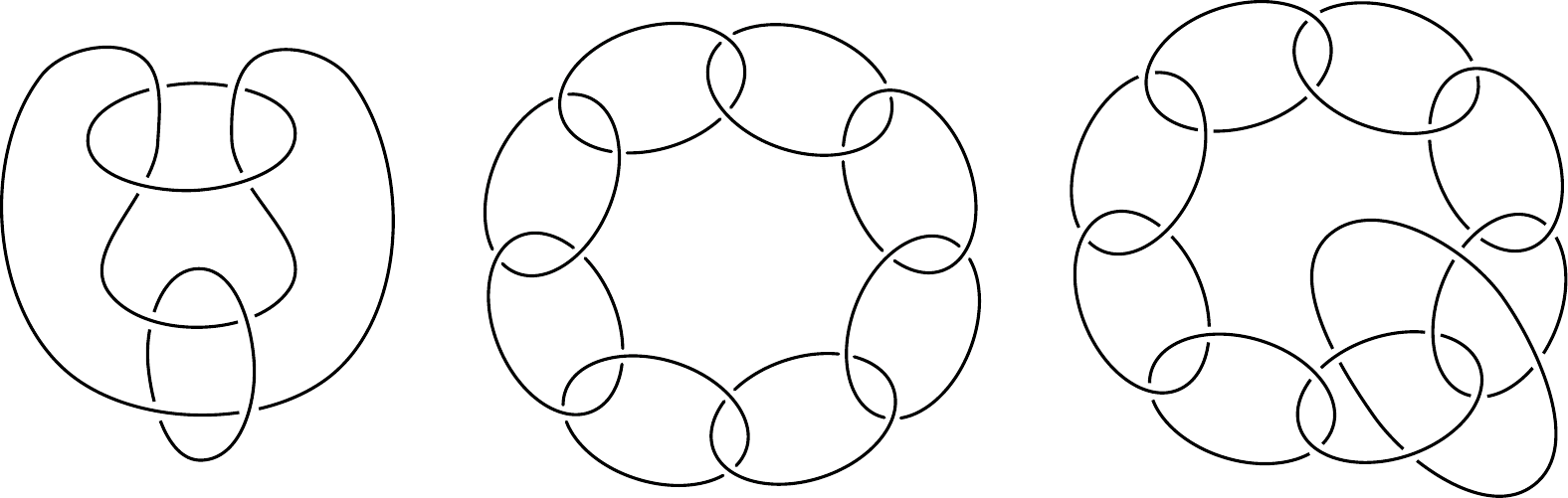}
		\put(0,28){$B$}
		\put(30,28){$P_4$}
		\put(66,28){$O_4$}
		\put(34,20){$c_{1}$}
		\put(50,25.5){$c_{2}$}
		\put(57.5,10.5){$c_{3}$}
		\put(41,4){$c_{4}$}
		\put(88,15){$c_{0}$}
	\end{overpic}
	\caption{On the left, the Borromean Rings, $B$, is depicted as a flat FAL. In the middle, $P_4$ is depicted with its four crossing circles labeled. On the right, $O_4$ is depicted, which can be constructed by adding the crossing circle $C_0$ to $P_4$.}
	\label{fig:BP4O4}
\end{figure}

We would like to understand how $\phi: Aut_{p}(\Gamma(L)) \rightarrow Sym^{+}(\Sth, L)$ behaves in this setting since such links admit type-changing symmetries, and so, the monomorphism $\phi$ won't be onto in these cases. In addition, previous work of the authors \cite[Theorem 3.14]{MiTr2023} shows that $Sym(\Sth, L) \cong Sym(\Sth \setminus L)$ when $L$ is a flat FAL whose complement admits multiple reflection surfaces. Thus, understanding the map $\phi$ in this setting helps us understand how well $Aut_{p}(\Gamma(L))$ represents symmetries of both the corresponding flat FAL and its complement.


\subsection{The Borromean Rings}
\label{subsec:Brings}

The Borromean rings, $B$, are the only flat FAL whose complement admits three distinct reflection surfaces. Any two of the three components of $B$ can play the role of the two crossing circles in a flat FAL diagram for $B$. Thus, there exist three distinct flat FAL diagrams for $B$, which correspond with three painted crushtaceans, all of which are isomorphic to the painted graph in Figure \ref{fig:CrushBor}.  SnapPy shows that $Sym(\Sth, B) \cong Sym(\Sth \setminus B) \cong G \times \mathbb{Z}_{2}$, where $G$ represents the (orientation-preserving) symmetries of the octahedron and $|G| = 24$.  Thus, $|Sym^{+}(\Sth, B)| = 24$, and we see that $Sym^{+}(\Sth, B)$ is isomorphic to the group of orientation-preserving symmetries of the octahedron, or equivalently, the cube. This show $Sym^{+}(\Sth, B)$ is isomorphic to a finite subgroup of $O(3)$.

	\begin{figure}[ht]
	\centering
	\begin{overpic}[scale=0.40]{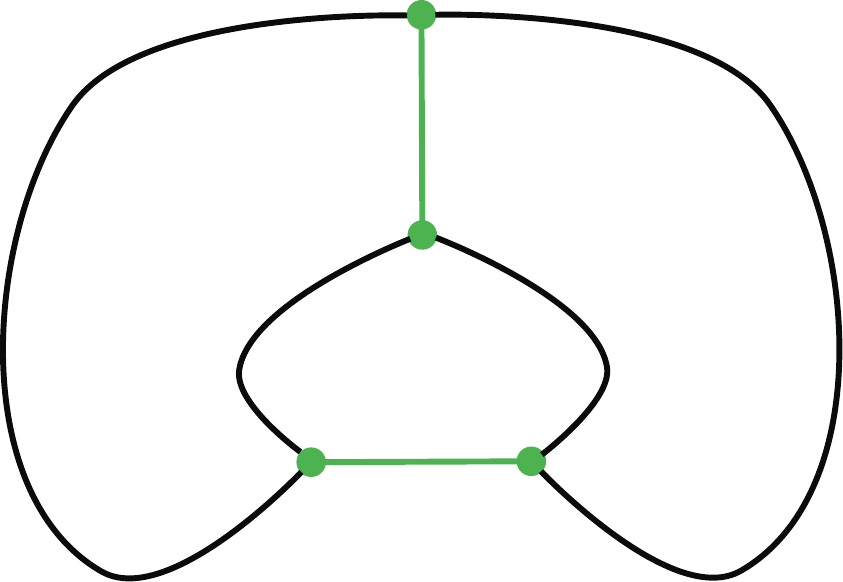}
		\put(52, 52){$p_2$}
		\put(48,8){$p_1$}
	\end{overpic}
	\caption{A crushtacean $\Gamma(B)$ for the Borromean rings, $B$. Painted edges are colored green and labeled.}
	\label{fig:CrushBor}
\end{figure}

 Assume $\Gamma(B)$ is a regular tetrahedron and let $S$ be the square whose vertices are midpoints of the unpainted edges.  Since $Aut_{p}(\Gamma(B))$ preserves unpainted edges, it preserves $S$ and $Aut_{p}(\Gamma(B)) \le D_4$. Moreover, reflection across the plane containing $p_1$ and intersecting $p_2$ at its midpoint acts as a reflection on $S$, while one can choose an element of $Aut_{p}(\Gamma(B))$ that swaps the painted edges and acts as a $\pi/2$ rotation on $S$. Thus $Aut_{p}(\Gamma(B))\cong D_4$. Note that,  $[Sym^{+}(\Sth, B) : \phi(Aut_{p}(\Gamma(B)))] =3$.


\subsection{Flat Fully Augmented Pretzel Links}
\label{subsec:Pn}

For $n \geq 3$, let $P_{n}$ represents a flat FAL obtained from augmenting a pretzel link with $n$ twist regions and undoing all half-twists within each twist region. Thus, $P_{n}$ has $2n$ components where $n$ of these components are crossing circles and $n$ are knot circles for any fixed flat FAL diagram for $P_n$. There are two possible flat FAL diagrams for $P_n$; crossing circles and knot circles switch roles between these diagrams. For $n \geq 3$, these links are known to be hyperbolic. In addition, such links admit symmetric flat FAL diagrams where it is easy to visualize all elements of $Sym^{+}(\Sth, P_n)$; see Figure \ref{fig:SymDiag} where two of the (orientation-preserving) symmetries, $\beta$ and $\gamma$, are explicitly highlighted. One can also see that this link admits an orientation-preserving symmetry $\alpha$  that moves each component in this chain to its clockwise neighbor by rotating the chain and then performing  a $\pi/2$ rotation in the direction of $\beta$. In fact, $\alpha$, $\beta$, and $\gamma$ generate $Sym^{+}(\Sth, P_n)$.  

The theorem below summarizes the classification of the symmetry groups of $P_{n}$ and their complements which mostly comes from the combined work of Meyer--Millichap--Trapp \cite{MeMiTr2020},  Millichap--Trapp \cite{MiTr2025}, and Belleman \cite{Belleman2023}. Briefly, when $n$ is even, we have that $\alpha^{2n} =1$, $<\alpha, \gamma> \cong D_{2n}$, $\beta$ is order two, and  $\beta$ commutes with both $\alpha$ and $\gamma$. When $n$ is odd we have $\alpha^{2n} = \beta$, and so, $Sym^{+}(\Sth, P_n)$ is generated by just $\alpha$ and $\gamma$. Furthermore, $|<\alpha>|=4n$,  $|<\gamma>| =2$, and $(\alpha \cdot \gamma)^{2} = 1$. The only case not covered by previous work is $P_3$. SnapPy shows that $Sym(\Sth, P_3) \cong D_3 \times D_4$ and $Sym(\Sth \setminus P_3) \cong D_4 \times G$, where $G$ represents the group of orientation-preserving symmetries of the octahedron. The groups $Sym^{+}(\Sth, P_3)$ and $Sym^{+}(\Sth \setminus P_3)$ are index two in their respective full symmetry groups. In particular, $|Sym^{+}(\Sth, P_3)| = 24$ and $|Sym^{+}(\Sth \setminus P_3)| = 96$. Since $\{\alpha, \gamma\}$ generate a subgroup of order $24$ for $Sym^{+}(\Sth, P_3)$, we see that $Sym^{+}(\Sth, P_3) = <\alpha, \gamma> \cong D_{12}$.

		\begin{thm}
	\label{thm:Pnsym}
	For $n \geq 4$, we have $Sym^{+}(\Sth \setminus P_n) \cong Sym^{+}(\Sth, P_n)$. Furthermore, for $n \geq 3$, we have $Sym^{+}(\Sth, P_n)$ is order $8n$ and 
	
	\begin{itemize}
		\item if $n$ is even, then $Sym^{+}(\Sth, P_n) \cong D_{2n} \times \mathbb{Z}_{2}$ with  generating set $\{\alpha, \beta, \gamma\}$ and
		\item if $n$ is odd, then $Sym^{+}(\Sth, P_n) \cong D_{4n}$ with generating set $\{\alpha, \gamma\}$.
	\end{itemize}
		\end{thm}

		\begin{figure}[ht]
			\centering
			\begin{overpic}[scale=0.70]{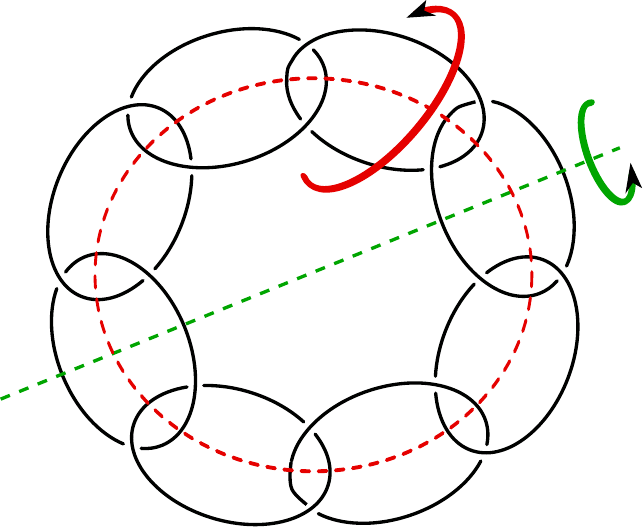}
				\put(73, 77){\LARGE{$\beta$}}
				\put(98,59){\LARGE{$\gamma$}}
			\end{overpic}
			\caption{A symmetric diagram for $P_{4}$}
			\label{fig:SymDiag}
		\end{figure}
		
In order to determine $Aut_{p}(\Gamma(P_n))$, we first note that Mani's polyhedron $P(P_n)$ in this case is a regular $n$-gonal prism.  For $n \ne 4$, the regular $n$-gon faces of $P(P_n)$ are distinct from the square faces and $Aut_{p}(\Gamma(P_n))\cong Aut(\Gamma(P_n))$.  As in Section \ref{subsec:finitesubO3} we see, for $n\ne 4$, that $Aut(\Gamma(P_n))$ is isomorphic to $ D_n\times \mathbb{Z}_2$ for $n$ even and $D_{2n}$ for $n$ odd.  For $n = 4$, we claim that $Aut_{p}(\Gamma(P_4)) \cong  D_4\times \mathbb{Z}_2$ as in the more general even case.  To see this, note that the polyhedron $P(P_4)$ is a cube and we assume painted edges join the top and bottom faces. Then $Aut_{p}(\Gamma(P_4))$ preserves these faces set-wise, and the subgroup $G$ that preserves the top must be a isomorphic to a subgroup of $D_4$.  It is not difficult to construct symmetries of $P(P_4)$ that realize every element of $D_4$, so $G\cong D_4$.  Moreover, the reflection $R_h$ of $P(P_4)$ across a horizontal plane through it's center swaps top and bottom while preserving painted edges, so $h \in Aut_{p}(\Gamma(P_4))$. One sees that $Aut_{p}(\Gamma(P_4))$ is generated by $R_h$ and a set of generators of $G$, and that $R_h$ commutes with elements of $G$, so that $Aut_{p}(\Gamma(P_4)) \cong D_4\times \mathbb{Z}_2$. 


We now explicitly describe the image of $\phi\left(Aut_{p}(\Gamma(P_n))\right)$ in $Sym^{+}(\Sth, P_n)$.  As above, let $R_h$ be reflection that swaps the $n$-gon faces of $P(P_n)$.  In addition, let $S$ be the $2\pi/n$-rotation about the prism's axis and $T$ be reflection across the plane containing the prism's axis and edge $p_1$. We see that $\phi(S) = \beta \cdot \alpha^{2}$, $\phi(R) = \beta$, and $\phi(T) = \gamma$. Note that, for $n \geq 3$, we have $[Sym^{+}(\Sth, P_n) : \phi(Aut_{p}(\Gamma(P_n)))] =2$. Since elements of $Aut_{p}(\Gamma(P_n))$ must map painted edges to painted edges, any corresponding element in $Sym^{+}(\Sth, P_n)$ must map crossing circles to crossing circles. So, $\alpha \in Sym^{+}(\Sth, P_n)$ is not induced by an element of $Aut_{p}(\Gamma(P_n))$.

		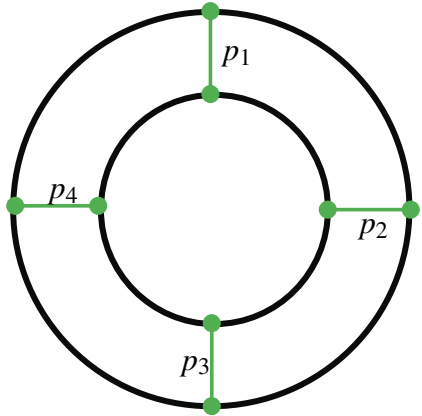
\begin{figure}[ht]
			\centering
			\begin{overpic}[scale=.80]{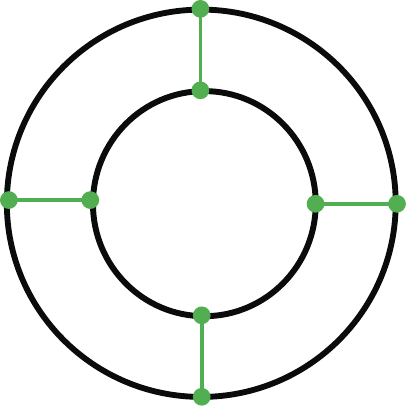}
				\put(52, 86){$p_{1}$}
				\put(85, 44){$p_{2}$}
				\put(42, 11){$p_{3}$}
				\put(10, 52.5){$p_{4}$}
			\end{overpic}
			\caption{A crushtacean $\Gamma(P_4)$ associated to the pretzel FAL $P_{4}$. Painted edges are green and labeled.}
			\label{fig:Crushtacean}
		\end{figure}


\subsection{Octahedral flat FALs with two reflection surfaces}
\label{subsec:On}

For $n \geq 2$, let $O_n$ represent the $2n+1$ component link obtained by augmenting the chain link $P_n$ with a single circle enclosing this chain, as shown in the right side of Figure \ref{fig:BP4O4}. 


\begin{thm}\label{thm:octsym}
	For $n \geq 2$, we have $Sym^{+}(\Sth \setminus O_n) = Sym^{+}(\Sth, O_n)$, with order $8n$. Furthermore, 
	
	\begin{itemize}
		\item if $n$ is even, then $Sym^{+}(\Sth, O_n) \cong D_{2n} \times \mathbb{Z}_{2}$ with  generating set $\{\alpha, \beta, \gamma\}$ and
		\item if $n$ is odd, then $Sym^{+}(\Sth, O_n) \cong D_{4n}$ with generating set $\{\alpha, \gamma\}$.
	\end{itemize}
	\end{thm}

\begin{proof}
 For $n \geq 2$, we have $Sym^{+}(\Sth \setminus O_n) = Sym^{+}(\Sth, O_n)$ by \cite[Theorem 3.14]{MiTr2023}. We treat $n=2$ as a special case, and SnapPy confirms that $Sym(\Sth, O_2) \cong Sym(\Sth \setminus O_2)$ has order $32$, and so, $Sym^{+}(\Sth, O_2)$ is order $16$. From our symmetric diagram for $O_2$, one can easily see that there is subgroup of $Sym^{+}(\Sth, O_2)$ of order $16$ that is isomorphic to $D_4 \times \mathbb{Z}_2$, and so, $Sym^{+}(\Sth \setminus O_2) = Sym^{+}(\Sth, O_2) \cong D_4 \times \mathbb{Z}_2$.  
 
 For $n \geq3$, we first justify that the component $C_0$ is preserved under any $f \in Sym^{+}(\Sth, O_n)$. Let $T_{0}$ represent the torus cusp  in $\Sth \setminus O_n$ corresponding to crossing circle $C_{0}$. Let $f \in Sym^{+}(\Sth, O_n) \leq Sym^{+}(\Sth \setminus O_n)$. We first claim  $f(T_0) = T_0$ and similarly for the corresponding link component $C_0$. By Mostow-Prasad rigidity, there is an isometry $\varphi_{f}: \Sth \setminus O_n \rightarrow \Sth \setminus O_n$ induced by $f$. By the work of Millichap--Trapp, $\Sth \setminus O_n$ admits precisely two distinct reflection surfaces, $R_1$ and $R_2$, and each such reflection surface contains a unique $(n+2)$-punctured sphere component, $S_1$ in $R_1$ and $S_2$ in $R_2$; see the proof of \cite[Theorem 3.14]{MiTr2023}. Note that, $C_0$ is the only component of $O_n$ with two punctures on $S_1$, and similarly for $S_2$. Thus, if $\varphi_{f}(R_1) = R_1$,  then $\varphi_{f}(S_1) = S_1$, and so, $f(T_0) = T_0$. Similarly, if $\varphi_{f}(R_1) = R_2$, then $\varphi_{f}(S_1) = S_2$, and so, $f(T_0) = T_0$. Thus, the cusp $T_0$ is preserved under any symmetry of $\Sth \setminus O_n$, and as a result, the component $C_0$ is preserved under any symmetry of $(\Sth, O_n)$. 
 
We now justify that $Sym^{+}(\Sth, O_n) \cong Sym^{+}(\Sth, P_n)$, which by Theorem \ref{thm:Pnsym} will complete the proof. Let $f \in Sym^{+}(\Sth, O_n)$. Then since $f$ preserves $C_0$ and $P_n = O_n \cup C_0$, we see that $f$ induces an orientation-preserving symmetry of $(\Sth, P_n)$. Thus, we have a well-defined homomorphism $\rho:Sym^{+}(\Sth, O_n) \rightarrow Sym^{+}(\Sth, P_n)$. Suppose $\rho(f) = I$ is the identity symmetry for $(\Sth, P_n)$. Then the corresponding isometry $\varphi_{I}: \Sth \setminus P_n \rightarrow \Sth \setminus P_n$ is the identity on the two reflection surfaces of $\Sth \setminus P_n$, and so, the isometry $\varphi_{f}: \Sth \setminus O_n \rightarrow \Sth \setminus O_n$ is also the identity on its pair or reflection surfaces, $R_1$ and $R_2$. Since $R_1$ and $R_2$ meet orthogonally, we see that $\varphi_{f}$ is an isometry that fixes any point in $R_1 \cap R_2 \subset \Sth \setminus P_n$  along with a tangent frame at that point. Proposition A.2.1 from Benedetti--Petronio \cite{BP} implies that $\varphi_{f}$ is the identity map. Thus, $f$ is the identity symmetry, and so, $\rho$ is injective, which implies that $\rho(Sym^{+}(\Sth, O_n))$ is a subgroup of $Sym^{+}(\Sth, P_n)$. However,  we see that $\alpha, \beta,$ and $\gamma$, the generators of $Sym^{+}(\Sth, P_n)$, all represent self-homeomorphisms of $(\Sth, O_n)$ up to isotopy. This is easy to see in the symmetric diagram of $O_n$ and one only needs to track that the crossing circle $C_0$ can be isotoped to itself in $\Sth \setminus P_n$ after performing these maps.  Thus, $\rho$ is an isomorphism and $G = <\alpha, \beta, \gamma> \cong Sym^{+}(\Sth, O_n)$, where $G$ has the same group structure as given in Theorem \ref{thm:Pnsym}. 
\end{proof}

 We now argue that $Aut_{p}(\Gamma(O_n)) \cong \mathbb{Z}_{2} \times \mathbb{Z}_{2}$.  Note that, adding $C_{0}$ to $P_n$ ``breaks'' much of the symmetry we saw in $\Gamma(P_n)$ by introducing the edge $p_0$ depicted in Figure \ref{fig:CrrushOn}.  In the planar embedding of $\Gamma(O_n)$ there are two distinct triangular regions, $T_1$ and $T_2$, and $p_0$ is the only edge adjacent to both.  Thus $p_0$ is invariant under any element of $f\in Aut_{p}(\Gamma(O_n))$, and $f$ either fixes or swaps $T_1$ and $T_2$. Moreover, elements of $Aut_{p}(\Gamma(O_n))$ either fix or swap the two non-triangular regions $R_1, R_2$ containing $p_0$ in their boundary cycles.  Note that there is a unique element of $Aut_{p}(\Gamma(O_n))$ which realizes each choice of fixing or swapping the pairs $\left\{T_1, T_2\right\}$ and $\left\{R_1, R_2\right\}$. Let $f_t$ swap the triangles while leaving the $R_i$ invariant, and let $f_r$ swap the $R_i$ while leaving the triangles invariant.  A quick check shows that $f_t$ and $f_r$ commute, while generating $Aut_{p}(\Gamma(O_n))$, and we see that $Aut_{p}(\Gamma(O_n)) \cong \mathbb{Z}_{2} \times \mathbb{Z}_{2}$, as desired.  

With this notation, observe that $\phi\left(f_t\right) = \gamma$ and $\phi\left(f_r\right) = \beta$.  In addition, for $n \geq 3$, we have $[Sym^{+}(\Sth, O_n) : \phi(Aut_{p}(\Gamma(O_n)))] =2n$. Thus, it is possible for this index to be arbitrarily large when a flat FAL $L$ admits a complement with multiple reflection surfaces.


	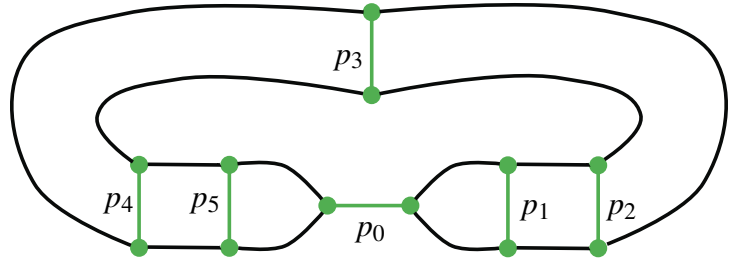
\begin{figure}[ht]
	\centering
	\begin{overpic}[scale=.80]{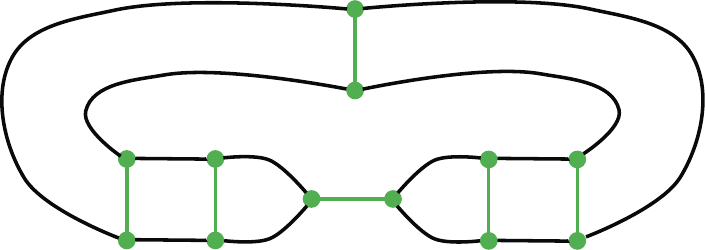}
		\put(48, 3){$p_{0}$}
		\put(71,6){$p_{1}$}
		\put(83, 6){$p_{2}$}
		\put(45,27){$p_{3}$}
		\put(13,7){$p_{4}$}
		\put(25,7){$p_{5}$}
	\end{overpic}
	\caption{The crushtacean $\Gamma(O_5)$. Painted edges are green and labeled.}
	\label{fig:CrrushOn}
\end{figure}


\section{Conclusion}
\label{sec:Conclusion}

We are now in a position to justify Theorem \ref{thm:main2} from Section \ref{sec:intro}. 

\begin{named}{Theorem \ref{thm:main2}}
	Let $G$ be an abstract group. Then $G \cong Sym^{+}(\Sth, L)$ for some b-prime flat FAL $L$ if and only if $G$ is isomorphic to a  finite subgroup of $O(3)$. Furthermore, for every finite $G \leq O(3)$, there exist an infinite class of distinct b-prime flat FALs $\{L_i\}$ such that $Sym^{+}(\Sth \setminus L_i) \cong Sym^{+}(\Sth, L_i) \cong G$, for all $i$.
\end{named}

\begin{proof}
Suppose $L$ is a $b$-prime flat FAL whose complement admits a unique reflection surface. Then Corollary \ref{cor:main1} and Theorem \ref{thm:O(3)isometries} show that $Sym^{+}(\Sth, L)$ must be a finite subgroup of $O(3)$. At the same time, Corollary \ref{cor:InfSysComp2} shows that for any finite subgroup $G$ of $O(3)$, we can build an infinite class of distinct $b$-prime FALs $\{L_i\}$, all with a unique reflection surface, with $Sym^{+}(\Sth \setminus L_i) \cong Sym^{+}(\Sth, L_i) \cong G$, for all $i$. 

Now suppose $L$ is a $b$-prime flat FAL whose complement admits multiple distinct reflection surfaces. It is easy to see that  each $O_n$ with $n \geq 2$ is $b$-composite. For instance, the crossing circle $C_0$ of $O_n$ admits multiple distinct crossing disks, and so, by Theorem 5.2 of \cite{MRSTZ2025}, $O_n$ is $b$-composite. For $n \geq 3$, each $P_n$ is a $b$-prime flat FALs that admits two distinct reflection surfaces. However, Theorem \ref{thm:Pnsym} shows that $Sym^{+}(\Sth, P_n)$ is isomorphic to either $D_{2n} \times \mathbb{Z}_{2}$ when $n$ is even or $D_{4n}$ when n is odd. Such groups are isomorphic to finite subgroups of $O(3)$, as needed. It only remains to consider the Borromean rings, $B$.  In \cite{MRSTZ2025}, the Borromean rings are considered a special case that are neither $b$-prime nor $b$-composite since they admit nontrivial belted-sum decompositions, but the summands are not flat FALs. This covers all the cases of flat FALs that admit multiple distinct reflection surfaces. 
\end{proof}  

Our work tells a complete story for (orientation-preserving) symmetry groups of $b$-prime flat FALs. However, a few questions remain.

\begin{question}
	Can one classify $Sym^{+}(\Sth, L)$ when $L$ is a $b$-composite flat FAL whose complement admits a distinct reflection surface? In this case, is  $Sym^{+}(\Sth, L)$ still isomorphic to a finite subgroup of $O(3)$?
\end{question}

In light of Corollary \ref{cor:main1}, it would be nice to have a simple way to determine if a $b$-prime flat FAL is a signature link.

\begin{question}
	Is there a simple combinatorial characterization for a painted crushtacean $\Gamma$ that determines if the corresponding flat FAL is a signature link?
\end{question}

	\bibliographystyle{hamsplain}
	\bibliography{biblio}
	
\end{document}